\documentclass[11pt,reqno]{amsart}

\usepackage{amsmath,amsthm,amsfonts}
\usepackage{latexsym,pxfonts}
\usepackage{amssymb}
\usepackage{color}

\newcommand{\mcN}{\mathcal{N}}
\newcommand{\mcS}{\mathcal{S}}
\newcommand{\mcC}{\mathcal{C}}
\newcommand{\mcB}{\mathcal{B}}
\newcommand{\mcJ}{\mathcal{J}}
\newcommand{\mcA}{\mathcal{A}}
\newcommand{\mcX}{\mathcal{X}}
\newcommand{\mcY}{\mathcal{Y}}
\newcommand{\mcK}{\mathcal{K}}
\newcommand{\mbx}{\underline{x}}
\newcommand{\mbe}{\underline{e}}
\newcommand{\mcH}{\mathcal{H}}
\newcommand{\mcD}{\mathcal{D}}

\newcommand{\LL}{\mathcal{L}}
\def\cL{\LL}
\newcommand{\HH}{\mathcal{H}}

\def\f{\frac}

\newcommand{\C}{\mathbb{C}}
\newcommand{\N}{\mathbb{N}}
\newcommand{\R}{\mathbb{R}}
\newcommand{\Z}{\mathbb{Z}}

\newcommand{\al}{\alpha}

\newcommand{\de}{\delta}

\newcommand{\fy}{\varphi}
\newcommand{\om}{\omega}
\newcommand{\la}{\lambda}

\newcommand{\p}{\partial}

\newcommand{\I}{\infty}

\newcommand{\ti}{\widetilde}

\newcommand{\EQ}[1]{\begin{equation}\begin{split} #1 \end{split}\end{equation}}
\newcommand{\Del}[1]{}

\def\ti{\tilde}

\numberwithin{equation}{section}

\newtheorem{thm}{Theorem}[section]
\newtheorem{cor}[thm]{Corollary}
\newtheorem{lem}[thm]{Lemma}
\newtheorem{prop}[thm]{Proposition}

\theoremstyle{remark}
\newtheorem{rem}{Remark}[section]
\newtheorem{defn}{Definition}[section]

\def\eps{\varepsilon}
\def\nn{\nonumber}

\def\lam{\lambda}

\def\lan{\langle}
\def\ran{\rangle}
\def\cD{\mathcal{D}}
\def\const{\mathrm{const}}

\begin{document}

\title[Exotic blowup solutions]{Exotic blowup solutions for the $u^5$ focusing wave equation in~$\R^3$}

\author{Roland Donninger, Min Huang, Joachim Krieger, Wilhelm Schlag}



\thanks{Support of the National Science Foundation  DMS-0617854, DMS-1160817 for the fourth author, and  the Swiss National Fund for
the third author  are gratefully acknowledged. The latter would like to thank the University of Chicago for its hospitality in August 2012}

\begin{abstract}
For the critical focusing wave equation $\Box u = u^5$ on $\R^{3+1}$ in the radial case, we construct a family of
blowup solutions which are obtained from the stationary solutions $W(r)$ by means of a
dynamical rescaling $\lambda(t)^{\frac12} W(\lambda(t)r)+\text{\tt correction}$ with $\lambda(t)\to\I$ as $t\to 0$.
The novelty here lies with the scaling law $\lambda(t)$ which
eternally oscillates between various pure-power laws. 
\end{abstract}

\maketitle

\section{Introduction}

The energy critical focusing wave equation in $\R^3$
\begin{equation}\label{eq:foccrit}
\Box u = u^5,\,\Box = \partial_t^2 - \triangle
\end{equation}
has been the subject of  intense investigations in recent years.
This equation is known to be locally well-posed in the space $\HH:=\dot H^{1}\times L^{2}(\R^{3})$,
meaning that if $(u(0),u_{t}(0))\in\HH$, then there exists a solution locally in time and continuous in time
taking values in~$\HH$.
Solutions need to be interpreted in the Duhamel sense:
\EQ{
u(t) = \cos(t|\nabla|)f + \frac{\sin(t|\nabla|)}{|\nabla|} g + \int_{0}^{t} \frac{\sin((t-s)|\nabla|)}{|\nabla|} u^{5}(s)\, ds
}
These solutions $\cL_{quintic}(u)=0$  have finite energy:
\[
E(u,u_t) = \int_{\R^{3}} \big[\frac12(u_t^2+|\nabla u|^2)-\f{u^6}{6}\big]\, dx =\const
\]
 The  recent series of papers \cite{DKM1} -- \cite{DKM4} establishes a complete classification of all {\it{possible}} type-II blow up dynamics. It remains, however, to investigate  the {\em existence}
of {\em all} allowed scenarios in this classification. Steps in this direction were undertaken in~\cite{KS}, \cite{KST}, \cite{DonKri12}, where a constructive approach to actually exhibit and thereby prove the existence of such type-II dynamics was undertaken. Recall that a type-II blow up solution $u(t, x)$ with blowup time $T_*$ is one for which
\[
\limsup_{t\rightarrow T_{*}}\|u(t, \cdot)\|_{H^1} + \|u_t(t, \cdot)\|_{L_x^2}<\infty
\]
but of which no extension in the usual sense of well-posedness theory in $\dot H^1\times L^2$ exists beyond time~$T_*$.
In \cite{DKM4}, it is demonstrated that such solutions can be described as a sum of dynamically re-scaled ground states
\[
\pm W(x) = \pm \big(1+\frac{|x|^2}{3}\big)^{-\frac{1}{2}}
\]
plus a radiation term. In particular, for solutions where only one such bulk term is present, one can write the solution as
\begin{equation}\label{eq:typeII}
u(t, x) = W_{\lambda(t)}(x) + \eps(t,x) + o_{\dot{H}^1}(1),\,W_{\lambda}(x) = \lambda^{\frac{1}{2}}W(\lambda x),\,\eps(t,\cdot)\in \dot{H}^1
\end{equation}
where the error is in the sense as $t\rightarrow T_*$. Moreover,  we have the dynamic condition
\begin{equation}\label{eq:rate}
\lim_{t\rightarrow T_*}(T_*-t)\lambda(t) = \infty
\end{equation}
In \cite{KST}, it was shown that such solutions with $\lambda(t) = t^{-1-\nu}$ do exist, where $\nu>\frac{1}{2}$ is arbitrary.
In~\cite{KrSch} the latter condition was relaxed to~$\nu>0$.

It is natural to ask which rescaling functions are admissible for~\eqref{eq:typeII} -- both in general,  and in
particular  within the confines of the method developed in~\cite{KST}, \cite{KrSch}.
It seems very difficult (perhaps hopeless) to answer this question in full generality.  Nevertheless, important progress
has been made in recent years such as in the deep work of Rapha\"el, Rodnianski~\cite{RR}, and Hillairet, Rapha\"el~\cite{HR}
who studied stable blowup laws (relative to a suitable topology) for energy critical equations.

The purpose of this paper is to exhibit
an uncountable family of rates which are not of the pure-power type as above. Our main result, which is in the spirit of \cite{KST0, KST, KST2},  is as follows.

\begin{thm}\label{Main}
Let  $\nu>3$ and $|\eps_0|\ll 1$ be arbitrary and define
\EQ{\label{Min0}
\la(t):=t^{-1-\nu} \exp( -\eps_0 \sin(\log t)),\quad 0<t<\f12
}
Then there exists a radial energy
solution $u$ of \eqref{eq:foccrit} which blows up precisely at
$r=t=0$  and which has the following property: in the cone $|x|=r\le
t $ and for small times $t$ the solution has the form
\EQ{\label{uWeta}
u(t,r) = \lambda^\frac12(t) W(\lambda(t) r) + \eta(t,x)
}
where $$\int_{[|x|<t]}  \big[|\nabla \eta(t,x)|^2+ |\eta_t(t,x)|^2 +
|\eta(t,x)|^6\big]\, dx   \to0\text{\ \ as\ \ }t\to0$$ and outside
the cone $u(t,r)$ satisfies
\[
\int_{[|x|\ge t]} \big[|\nabla u(t,x)|^2+ |u_t(t,x)|^2 +
|u(t,x)|^6\big]\, dx <\delta
\]
for all sufficiently small $t>0$ where $\de>0$ is arbitrary but fixed. In particular, the energy of these
blow-up solutions can be chosen arbitrarily close to~$E(W,0)$,
i.e., the energy of the stationary solution.
\end{thm}

We remark that
\[
\la(t)=t^{-1-\nu(t)},\quad \nu(t)=\nu + \eps_0\frac{\sin(\log t)}{\log t} \to\nu \text{\ \ as\ \ } t\to0+
\]
This shows that $\nu(t)$ eternally oscillates around the constant $\nu$, but does approach  that constant.
Currently we do not know if it is possible to have such solutions for which $\nu(t)$ is not asymptotically constant.

The starting point of our investigation was to adapt the method from~\cite{KST} to the setting where $\la(t)$
is not restricted to the class of pure-power laws. This turns out to run into serious difficulties essentially
from the beginning, with the ``renormalization construction" of the approximate solution being the first
serious obstacle. Recall from~\cite{KST} that this construction relies on an iterative procedure and involves
delicate book-keeping of various asymptotic expansions of the approximate solutions, the corrections, as
well as the errors. For the more general rates $\la(t)$ this cannot be done in the same fashion, and we succeeded
in a much modified fashion for the laws~\eqref{Min0}; however, only two steps of the iteration seem feasible.
This then forces one to confront a very major difficulty which was not present in~\cite{KST}; namely the lack
of a suitable smallness parameter which allowed for the ultimate contraction argument yielding an {\em exact}
solution rather than an approximate one to go through. In absence of this small parameter we are forced
to follow a different route. The idea is very simple but its actual implementation turns out
to be quite
subtle. Schematically, we have to deal with
a fixed point problem on a Banach space of the form
\[ x=F(x)+Ax + x_0 \]
where the norm of the linear operator $A$ is not small.
However, it turns out that $A^n$ has small operator norm provided $n$ is sufficiently large.
This implies the existence of $(1-A)^{-1}$ (via the Neumann series) and thus, we may re-write
the problem at hand as
\[ x=(1-A)^{-1}F(x)+(1-A)^{-1}x_0 \]
which we then solve by the Banach fixed point theorem.
Thus, a large part of the present paper is devoted to the development of a technique that allows
one to show smallness of $\|A^n\|$. In order to succeed, we have to exploit the fine-structure
of the operator $A$, in particular smoothing properties and oscillations.

 \section{The approximate solution for a modified power-law rescaling}
 \label{sec:approx}

\subsection{Generalities}

The radial quintic wave equation in $\R^3$ is
\EQ{\label{u5}
\cL_{quintic} u := u_{tt}-u_{rr}-\f2r u_r - u^5=0
}
A special stationary solution is $W(r)=(1+r^2/3)^{-\f12}$.  By scaling, $\la^{\f12}W(\la r)$ is also a solution for any $\la>0$.
We are interested in letting $\la$ depend on time. More precisely, we would like to find solutions $\cL_{quintic} u=0$
of the form
\EQ{\label{cusp}
u(t,r)= \la(t)^{\f12} W(\la(t) r) + \eps(t,r),\qquad \la(t)\to\infty \text{\ \ as\ \ }t\to0+
}
and $\eps$ small in a suitable sense. It suffices to show that $\eps$ remains small in energy, since  this ensures that the solution blows up at time $t=0$ by the mechanism
of ``energy concentration" at the tip of the light-cone $(t,r)=(0,0)$ (think of solving backwards in time).
In the paper \cite{KST}  such solutions were found with $\la(t)=t^{-1-\nu}$, and $\nu>\f12$ constant.
The goal here is to allow for more general functions; more specifically, we will set
\EQ{\label{Min}
\la(t)=t^{-1-\nu} \exp( -\eps_0 \sin(\log t)),\quad \nu>3,\; |\eps_0|\ll 1
}
This is of the form $\la(t)=t^{-1-\nu(t)}$ with
\[
\nu(t) = \nu + \eps_0\frac{\sin(\log t)}{\log t} \to\nu \text{\ \ as\ \ } t\to0+
\]
For future reference, we introduce $\mu(t):=t\la(t)$, and
\[
\kappa(t):= - \f{t\dot\mu(t)}{\mu(t)}\]
so that for \eqref{Min} we obtain
\[
\kappa(t)  = \nu + \eps_0\cos(\log t)
\]
Our goal is to prove the following result. In what follows, $R=r\la(t)$.

\begin{prop}\label{prop1}
Let $\la(t)$ be as in \eqref{Min} and $t_0\ll 1$.

(i) There exists some $u_2(t,r)\in C^2(\{0<t<t_0, 0\leqslant r\leqslant t\})$ such that
\EQ{\label{u2bd}
u_2(t,r) = \sqrt{\la(t)} (W(R) + \mu^{-2}(t)O(R)), \qquad 0<t<t_0, 0<r<t
}
and $e_2:=\cL_{quintic}u_2$ satisfies
\EQ{\label{e2bd}
t^2 \la^{-\f12}(t)e_2(t,r)=\mu^{-2}(t)O\left(\frac{\log (R+2)}{R+1}\right), \qquad 0<t<t_0, 0<r<t
}
(ii) For $0<t<t_0$, $0<r< t/2$ and all $k,j\geqslant 0$ we have
\EQ{\label{u2der}
\partial_t^k\partial_r^j  u_2(t,r) =\partial_t^k\partial_r^j \sqrt{\la(t)} W(R)+t^{-k}r^{-j}\sqrt{\la(t)}\mu^{-2}(t)O(R)  }
and
\EQ{\label{e2der}
\partial_t^k\partial_r^j  \left( t^2 \la^{-\f12}(t)e_2(t,r)\right) =t^{-k}r^{-j}\mu^{-2}(t)O\left(\frac{\log (R+2)}{R+1}\right)  }
The same bound applies without the restriction $r<\frac{t}{2}$, provided $k+j\leqslant 2$.

(iii) The function $u_2(t, r)$ admits a $C^2$- extension (on fixed time slices) beyond the light cone $r\leq t$ with the property that given $\delta>0$,
\[
\int_{r\geq t}[|\partial_r u_2|^2 + \partial_t u_2^2 + u_2^6](t,r)\, r^2\, dr<\delta
\]
provided $t<t_0$ is sufficiently small.

\end{prop}

The proof will be given in Section \ref{prop1proof}.

\subsection{The bulk term}
Define
\EQ{\label{u0def}
u_0(t,r)=\la(t)^{\f12} W(r\la(t)) = \la(t)^{\f12} W(R)
}
While $u_{0}$ is very far from being an approximate solution, the construction in \cite{KST} for $\la(t)=t^{-\alpha}$ where $\alpha>1$ is {\bf constant} shows  that one can add
successive corrections via an iterative procedure
\[
u=u_{0}+v_{1}+v_{2}+ v_{3} +\ldots +v_{k}
\]
so that this function approximately solves~\eqref{u5} in the light cone $\{r\le t\ll 1\}$.  To be specific, we  achieved that $\cL_{quintic} u(t)$ goes to zero like $t^{N}$ in the energy norm
as $t\to0$ where $N$ can be made arbitrarily large by taking~$k$ large.

For \eqref{Min} we will content ourselves with two steps of the construction
only, i.e., $u=u_{0}+v_{1}+v_{2}$.
  Let us first compute the error resulting from~$u_{0}$.

Define $\cD:=\f12+r\p_r=\f12+R\p_R$.  Then
\EQ{
e_0:=\cL_{quintic} u_0 &= \la^{\f12}(t) \Big [ \Big( \f{\la'}{\la}\Big)^2(t) (\cD^2 W)(R) + \Big(\f{\la'}{\la}\Big)^{\prime}(t) (\cD W)(R)  \Big] \\
t^2 e_0 &=: \la^{\f12}(t) \Big [ \om_1(t) \f{1-R^2/3}{(1+R^2/3)^{\f32}}   + \om_2(t) \f{ 9-30R^2+R^4}{(1+R^2/3)^{\f52} }  \Big] \label{e0}
}
We note that \eqref{Min} satisfies
\EQ{\label{ass1}
\f{t\la'(t)}{\la(t)},\quad  \f{t^2\la''(t)}{\la(t)}  = O(1),\quad t\to0+
}
and analogously for higher derivatives. Moreover, the functions on the left remain bounded under $(t\p_t)^\ell$ for any~$\ell$;
the same properties hold for $\omega_1(t), \omega_2(t)$.

Then $t^2 e_0 = \la(t)^{\f12}O(R^{2}\lan R\ran^{-3})$ uniformly in $0<t\ll 1$ (with derivatives).
Clearly, this error  blows up as $t\to0$ like $t^{-2}$.

  \subsection{The first correction}\label{sec:v1}

  Then $t^2 e_0 = \la(t)^{\f12}O(R^{2}\lan R\ran^{-3})$   as $R\to\I$.
This error   blows up as $t\to0$ like $t^{-2}$. The goal is now to reduce it --- in fact turn it into an error that vanishes as $t\to0$ ---  by adding corrections
to~$u_{0}$, the first one being~$v_{1}$. We will do this by setting   $\lam^2(t) L_0 v_1 = e_0$ where
\EQ{
L_0 &:= \p_R^2 + \f2R\p_R +5 W^4(R)
}
Note that this is the linearized operator obtained by plugging $u_{0}+v_{1}$ into~\eqref{u5} and discarding~$\p_{t}$ altogether.
While this may seem strange, the idea is to look first at the regime $0<r\ll t$ where $\p_{t}$ should matter less than~$\p_{r}$.
We shall see shortly that $v_{1}$ has the good property that it decays like $(t\la(t))^{-2}$, but it produces errors for the nonlinear PDE that grow in~$r$
too strongly. To remove this growth, we carry out a correction
at the second stage by solving a suitable differential operator. At this stage the self-similar variable~$a=\f{r}{t}$ becomes important.

Now we discuss~$v_{1}$ in more detail.
A fundamental system of $L_0$ is
\EQ{\label{L0 fund}
\fy_1(R):=\f{1-R^2/3}{(1+R^2/3)^{\f32}},\quad \fy_2(R):=\f{1-2R^2 +R^4/9}{R(1+R^2/3)^{\f32}}
}
The operator
\EQ{\label{tilde L0}
\tilde L_0 = R L_0\, R^{-1} = \p_R^2  +5 W^4(R)
}
has a fundamental system
\EQ{\label{tilde fundsys}
\tilde \fy_1(R) &:=\f{R(1-R^2/3)}{(1+R^2/3)^{\f32}} =  \tilde\psi_1 (R^{-2})\\
\tilde \fy_2(R) &:= \f{1-2R^2 +R^4/9}{(1+R^2/3)^{\f32}} = R \tilde\psi_2 (R^{-2})
}
The right-hand sides here are for large~$R$, and the $\tilde\psi_j$ are analytic around~$0$.
The Wronskian is
\EQ{\nn
\tilde\fy_1'(R)\tilde \fy_2(R)-\tilde \fy_1(R)\tilde \fy_2'(R) = 1
}
We let $\mu(t)=t\la(t)$ as above, and
\EQ{
\mu^2(t) L_0 v_1 = t^2 e_0,\quad v_1(0)=v_1'(0)=0
}
We claim that
\EQ{\nn
v_1(t,r) =  \mu^{-2}(t) L_0^{-1} t^2  e_0  =  \la^{\f12}(t)  \mu^{-2}(t) O(R) \text{\ \ as\ \ }R\to\I
}
To be more specific,
write
\EQ{
t^2\, e_0 = \la^{\f12}(t)\big( \om_1(t) \, g_1(R) + \om_2(t) \, g_2(R) \big)
}
see \eqref{e0}.  Note that the $g_j$ are of the form
\EQ{\label{gj analytic}
g_j(R)= R^{-1} \phi_j(R^{-2}) \quad R\gg 1
}
where $\phi_j$ is analytic around $0$.
Then  $L_0 f_j = g_j$ with $f_j(0)=f_j'(0)=0$ satisfies
\EQ{\label{int}
f_j(R) = R^{-1} \Big( \tilde\fy_1(R) \int_0^R \tilde\fy_2(R')R' g_j(R')\, dR' -  \tilde\fy_2(R) \int_0^R \tilde\fy_1(R')R' g_j(R')\, dR'\Big)
}
for $j=1,2$. Then one checks that
\EQ{\label{fasymp}
f_j(R) &= b_{1j} R + b_{2j} + b_{3j}\f{\log R}{R} + O(1/R)\text{\ \ as\ \ }R\to\I \\
f_j(R) &= c_{1j} R^2 +O(R^4) \text{\ \ as\ \ }R\to 0
}
In fact,   around $R=0$ the $f_j(R)$ are even analytic functions, whereas  around $R=\I$ one has the representation
\EQ{\label{fasymp infty}
f_j(R) &= R(b_{1j} + b_{2j} R^{-1}+ R^{-2}\log R\;\fy_{1j}(R^{-2}) +  R^{-2} \fy_{2j}(R^{-1})  ) \\
&=: R(F_j(\rho) + \rho^{2} G_{j}(\rho^2)\log \rho)
}
where $\fy_{1j}, \fy_{2j}$ and $F_j, G_j$ are analytic around zero, with $\rho:=R^{-1}$.  This follows from~\eqref{tilde fundsys}, \eqref{gj analytic}, and~\eqref{int}.
For future reference, we remark that the structure in~\eqref{fasymp infty} is preserved under application of~$\cD$.
In particular, and abusing notation somewhat, we have
\EQ{\label{v1w1}
v_1(t,r) = \la^{\f12}(t)  \mu^{-2}(t)  (\om_1(t) f_1(R) + \om_2(t) f_2(R)) =: \la^{\f12}(t) \mu^{-2}(t) \om(t)  f(R)
}
Define
$$u_1:= u_0+ v_1 = \la^{\f12}(t)\big( W(R) + \mu^{-2}(t) \om(t) f(R) )$$
In view of~\eqref{fasymp}, and $R\le \mu$ (recall that we are inside of the light cone $r\le t$)
\EQ{\label{u1 R}
u_1(t,r) = \la^{\f12}(t) O(R^{-1}) \quad R\geqslant1 \\
u_1(t,r) = \la^{\f12}(t) O(1) \quad 0\leqslant R<1
}
uniformly in $0<t<1$; moreover, we may apply $t\p_t$ or $r\p_r=R\p_R$
any number of times without affecting this asymptotic property.
Finally, $\la(t)^{-\f12} u_1(t,r)$ is an even analytic function around $R=0$.

 \subsection{The  error from $u_1$}
Set  $e_1:= \cL_{quintic} (u_1)$.
Then
\EQ{\label{e1def}
e_1 &=  \p_t^2 v_1 - 10 u_0^3 v_1^2 -10 u_0^2 v_1^3 -5 u_0 v_1^4 - v_1^5
}
One has
\EQ{\label{e1}
t^2 \la^{-\f12}(t) e_1 &= \la^{-\f12}(t) ((t\p_t)^2 - t\p_t) \big(  \la^{\f12}(t) w_1(t,r\lam(t)) \big) - \mu^2(t)  (10 W^3(R) w_1^2(t,R) \\
&\qquad + 10 W^2(R) w_1^3(t,R) + 5 W(R) w_1^4(t,R) + w_1^5(t,R))
}
We write symbolically $w_1(t,R)= \mu^{-2}(t) \om(t) f(R)$. Then the nonlinearity in~\eqref{e1} is
\EQ{
& \mu^2(t)(10 W^3(R) w_1^2(t,R)  + 10 W^2(R) w_1^3(t,R) + 5 W(R) w_1^4(t,R) + w_1^5(t,R))  \\
& = \mu^{-2}(t) \big(10 W^3(R) \om^2(t) f^2(R)   + 10 W^2(R) \mu^{-2}(t) \om^3(t)f^3(R) \\
&\quad + 5 W(R)  \mu^{-4}(t) \om^4(t)f^4(R) + \mu^{-6}(t) \om^5(t) f^5(R) \big)
}
whereas
\EQ{
& \la^{-\f12}(t) ((t\p_t)^2 - t\p_t) \big(  \la^{\f12}(t)w_1(t,r\lam(t)) \big) \\
&=\left(\left(t\p_t + \frac{t\la'(t)}{\la(t)}\cD\right)^2 - \left(t\p_t + \frac{t\la'(t)}{\la(t)}\cD\right) \right)w_1(t,R)
}
Now
\EQ{\label{mutd}
&\mu^2(t) \left(t\p_t + \frac{t\la'(t)}{\la(t)}\cD\right) \mu^{-2}(t) \om(t) f(R) \\&= \Bigg( -\f{2t\dot\mu(t)}{\mu(t)}\om(t) + t\dot\om(t)  + \frac{t\la'(t)}{\la(t)}\om(t) \cD \Bigg) f(R)
}
Note that this is schematically of the form $\om(t) f(R)$ with $f$ as in~\eqref{fasymp}, and $\om(t)$ bounded together with all powers of $t\p_t$ as $t\to0+$.
Henceforth, we refer to such functions $\om(t)$ as {\em admissible}.
Thus we can write
\EQ{\label{e1*}
&t^2 \la^{-\f12}(t) e_1(t,r) \\
&= \mu^{-2}(t) \Big( \om(t) f(R) - \big(10 W^3(R) \om^2(t) f^2(R)  + 10 W^2(R) \mu^{-2}(t) \om^3(t)f^3(R)  \\ &
 \qquad \qquad + 5 W(R)  \mu^{-4}(t) \om^4(t)f^4(R) + \mu^{-6}(t) \om^5(t) f^5(R) \big) \Big)
}
We let $a=\frac{r}{t}=\frac{R}{\mu}=Rb$, $b:=\mu^{-1}$ and isolate those terms   in~\eqref{e1*} which do not decay for large~$R$.
Since we are working inside of the light-cone, we have $0\le a\le 1$.
Now, abusing notation somewhat,
\begin{align}
\mu^{-2}(t)  f(R) &= b^2 R (F(\rho)+ \rho^{2} G(\rho^2)\log \rho)  = ba (F(\rho)+ \rho^{2} G(\rho^2)\log \rho)  \nn  \\
\mu^{-2}(t)  W^3(R)  f^2(R)  &=  b^2 R^{-3} \Omega(\rho^2) R^2 (F(\rho)+ \rho^{2} G(\rho^2)\log \rho)^2 \nn  \\
&=  b^2 R^{-1}    (F(\rho)+ \rho^{2} F(\rho)\log \rho + \rho^{4} G(\rho^2)\log^2 \rho ) \label{mess1} \\
\mu^{-4}(t)  W^2(R)  f^3(R)  &=  b^4 R^{-2}\Omega(\rho^2) R^3 (F(\rho)+ \rho^{2} G(\rho^2)\log \rho)^3 \nn \\
&=  b^3 a (F(\rho)+ \rho^{2} F(\rho)\log \rho + \rho^{4} F(\rho)\log^2 \rho + \rho^{6} G(\rho^2)\log^3 \rho)   \nn
\end{align}
where $F,G$ can change from line to line.
Similarly,
\EQ{ \label{mess2}
\mu^{-6}(t)  W(R)  f^4(R)  &= b^6 R^{-1} \Omega(\rho^2)  R^4 (F(\rho)+ \rho^{2} G(\rho^2)\log \rho)^4\\
&= b^3 a^3   (F(\rho)+ \rho^{2} F(\rho)\log \rho  + \rho^{4} F(\rho)\log^2 \rho \\
&\qquad + \rho^{6} F(\rho)\log^3 \rho + \rho^{8} G(\rho^2)\log^4 \rho) \\
\mu^{-8}(t)    f^5(R)  &= b^8 R^5  (F(\rho)+ \rho^{2} G(\rho^2)\log \rho)^5 \\
&= b^3 a^5 (F(\rho)+ \rho^{2} F(\rho)\log \rho  + \rho^{4} F(\rho)\log^2 \rho \\
&\qquad + \rho^{6} F(\rho)\log^3 \rho + \rho^{8} F(\rho)\log^4 \rho + \rho^{10} G(\rho^2)\log^5 \rho)
}
From \eqref{mess1}, \eqref{mess2} we extract the
leading order
\EQ{ \label{e10*}
t^2 \la^{-\f12}(t) e_1^0(t,r)  &:= \mu^{-1}(t) ( c_1 a + c_2 b+ ( c_3 a   + c_4 a^3   + c_5  a^5 )  b^2 )
}
with $c_j=c_j(t)$   admissible functions.
Indeed, from the first line in~\eqref{mess1} we extract $ba(F(0)+\rho F'(0))= ba c_1 +  b^2 c_2$,
whereas from the fifth we extract $b^3 a F(0)$. From the second line in~\eqref{mess2} we retain $b^3 a^3 F(0)$,
and from the fifth one~$b^3 a^5 F(0)$.
The point here is that with this choice of $e_1^0$ one obtains  a decaying error as $R\to\I$
\EQ{ \label{ediff}
&t^2 \la^{-\f12}(t) ( e_1- e_1^0)(t,r) \\
& = \mu^{-2}(t) \Big[ \frac{\log R}{R} \Phi_1(t,a,b,\rho\log\rho,\rho) + \frac{1}{R} \Phi_2(t,a,b,\rho\log\rho,\rho)\Big]
}
where $\Phi_j(t,a,b,u,v)$ are polynomials in $a,b$ and analytic in $u,v$ near $(0,0)$; moreover, their time dependence is
polynomial in admissible functions. Writing $b=\frac{a}{R}$ we may
delete the terms involving~$b^2=ba/R$ on the right-hand side of~\eqref{e10*}, since they are of the form~\eqref{ediff}.
Thus, it suffices to consider the simpler leading error
\EQ{ \label{e10}
t^2 \la^{-\f12}(t) e_1^0(t,r)  &:=c_1 a  \mu^{-1}(t) + c_2\mu^{-2}(t) = c_1(t) ab + c_2(t) b^2
}
with $c_1(t), c_2(t)$ admissible.

\subsection{The second correction}
Now we would like to solve the corrector problem ``near $r=t$", i.e.,
\EQ{
t^2 \big(v_{tt}-v_{rr}-\f2r v_r   \big) = -t^2 e_1^0 
}
Note that we have discarded the nonlinearity on the left-hand side since it decays near~$r=t$.
This is designed exactly so as to remove the growth in~$R$.
We seek a solution in the form
\EQ{\label{v2}
v(t,r) = \la(t)^{\f12}\big( \mu^{-1}(t) q_1(a,t) + \mu^{-2}(t) q_2(a,t) \big)
}
with boundary conditions $q_1(0,t)=0, q_1'(0,t)=0$ and $q_2(0,t)=0$, $q_2'(0,t)=0$.
These translate into the boundary conditions $v(t,0)=0$, $\p_r v(t,0)=0$ at $r=0$.  This $v$ will essentially be the
function~$v_2$.
In view of
\[
\la(t)^{-\f12} \mu^\al \p_t\:  \la(t)^{\f12} \mu^{-\al} =  \p_t + t^{-1} \big( \f12 \f{t\dot\la}{\la} - \al \f{t\dot\mu}{\mu}\big)
\]
we are reduced to the system
\EQ{\label{w1}
&t^2 \Big(-\big(\p_t + \f{\beta_1(t)}{t}\big)^2 +\p_{rr}+\f2r \p_r  \Big) q_1(a,t)  =   c_1(t) a }
and
\EQ{\label{w0}
&t^2 \Big(-\big(\p_t +  \f{\beta_2(t)}{t}\big)^2 +\p_{rr}+\f2r \p_r   \Big) q_2(a,t)  =   c_2(t)
}
where
\[
\beta_j(t) = \f12 \f{t\dot\la}{\la} - j \f{t\dot\mu}{\mu}=(j-1/2)\kappa(t) -\f12 ,\quad j=1,2.
\]
We impose the boundary conditions $q_j(0,t)=\p_a q_j(0,t)=0$.

\begin{lem}\label{lem1}
Let $\la(t)$ be as in \eqref{Min} with $|\eps_0|$ sufficiently small. Equations \eqref{w1}, \eqref{w0} have solutions $q_j(a,t)$ satisfying $q_j(0,t)=\partial_{a}q_j(0,t)=0$,
$q_j(a,t)\in C^2(\{0<t<t_0, 0\leqslant a\leqslant 1\})$. Furthermore,
for $k\geqslant 0$ and $0\leqslant \ell\leqslant 2$ we have
\EQ{\label{q1a3}
\partial_{t}^k\partial_{a}^\ell q_2(a,t)=O(t^{-k}a^{2-\ell}) ; \ \
\partial_{t}^k\partial_{a}^\ell q_1(a,t)=O(t^{-k}a^{3-\ell})
}
\end{lem}
\begin{proof}
Now, with $\dot q_j=\p_t q_j$, \eqref{w1} and \eqref{w0} can be written as
\EQ{\label{w1w0}
& t^2 \Big(-\big(\p_t + \f{\beta_j(t)}{t}\big)^2 +\p_{rr}+\f2r \p_r  \Big) q_j(a,t) \\
& = \big( (1-a^2)\p_a^2 + (2(\beta_j(t)-1)a+2a^{-1})\p_a - \beta_j^2(t) + \beta_j(t) -t \dot\beta_j(t) \big) q_j(a,t)\\
&\qquad   - (t^2\ddot  q_j(a,t)+2\beta_j(t) t\dot  q_j(a,t)) +2 at\p_a\dot q_j(a,t)=c_j(t)a^{2-j}
}
By \eqref{Min} we have 
$\beta_j(t)=\tilde{\nu}_j+2\tilde{\eps}_j\cos(\log t)$ with $\tilde{\nu}_j=(j-1/2)\nu -\f12>1$ and  $\tilde{\eps}_j=(j/2-1/4)\eps_0$. We note that the admissible functions $c_j$ are as in \eqref{e10*}, which are of the form $\om^k(t)$ in \eqref{e1*}. These are parts of $t^2 \la^{-\f12}(t) e_1(t,r)$ with $e_1$ defined in \eqref{e1def}, and they come from \eqref{mutd} applied no more than twice to $v_1$ instead of $f$, where $v_1$ is as defined in \eqref{v1w1} with $\om_j$ coming from $t^2 e_0$ in \eqref{e0}. Thus we see that $c_j(t)$ are polynomials of $\kappa(t)$ with the operator $t\p_t$ applied finitely many times, which means we can write
\begin{equation}
c_j(t)=\sum_{n=0}^{N_j}\sum_{m=0}^{n}\tilde{\epsilon}_j^{n}\tilde{c}^{(j)}_{n,m}\, t^{(n-2m)i}\label{eq:om024}
\end{equation}
since polynomials of $\kappa(t)$ have this type of expansions and they are preserved by the operator $t\p_t$.

For convenience we drop the subscript $j$ and write \eqref{w1w0} as
\EQ{\label{reducedw}
\big( (1-a^2)\p_a^2 + (2(\beta(t)-1)a+2a^{-1})\p_a - \beta^2(t) + \beta(t) -t \dot\beta(t) \big) q(a,t)\\
\qquad   - (t^2\ddot  q(a,t)+2\beta(t) t\dot  q(a,t)) +2 at\p_a\dot q(a,t)=c(a,t)
}
where $$\beta(t)=\tilde{\nu}+2\tilde{\eps}\cos(\log t)=\tilde\nu + \tilde\eps\, t^i + \tilde\eps \, t^{-i}$$ and
$$c(a,t)=\sum_{n=0}^{N}\sum_{m=0}^{n}\tilde{\eps}^{n}\tilde{c}_{n,m}(a)\, t^{(n-2m)i}$$
where $\tilde{c}_{n,m}(a)$ is linear in $a$. 
We seek a solution to \eqref{reducedw} of the form
\begin{equation}
q(a,t)=\sum_{n=0}^{\infty}\sum_{m=0}^{n}\tilde{\eps}^{n}g_{n,m}(a)\, t^{(n-2m)i}\label{eq:w0exp}
\end{equation}
where $g_{n,m}(a)=\bar{g}_{n,n-m}(a)$. Plugging (\ref{eq:w0exp})
and (\ref{eq:om024}) into (\ref{reducedw}) we see that
\begin{multline*}
\sum_{n=0}^{\infty}\sum_{m=0}^{n}\tilde{\eps}^{n}\, t^{(n-2m)i}\biggl((1-a^{2})g''_{n,m}(a)+(2(\tilde{\nu}-1)a+2a^{-1})g'_{n,m}(a)\\+2a(g'_{n-1,m}(a)+g'_{n-1,m-1}(a))
+(\tilde{\nu}-\tilde{\nu}^{2})g{}_{n,m}(a)-2g{}_{n-2,m-1}(a)\\+(1-i-2\tilde{\nu})g{}_{n-1,m}(a)+(1+i-2\tilde{\nu})g{}_{n-1,m-1}(a)-g{}_{n-2,m}(a)-g{}_{n-2,m-2}(a)\\
+(n-2m)(n-2m+i-2\tilde{\nu}i)g{}_{n,m}(a)-2i(n-2m-1)g{}_{n-1,m}(a)\\-2i(n-2m+1)g{}_{n-1,m-1}(a)+2ai(n-2m)g'_{n,m}(a)\biggr)
=\sum_{n=0}^{\infty}\sum_{m=0}^{n}\tilde{\eps}^{n}\tilde{c}_{n,m}(a)\, t^{(n-2m)i}
\end{multline*}
where $g_{n,m}(a)=0$ if $n<0$ or $|n-2m|>n$. Collecting powers
of $\tilde{\eps}$ and $t^{i}$ we obtain the equation
\begin{multline}
(1-a^{2})g''_{n,m}(a)+(2(\tilde{\nu}-1)a+2a^{-1}+2ai(n-2m))g'_{n,m}(a)\\+(\tilde{\nu}-\tilde{\nu}^{2}+(n-2m)(n-2m+i-2\tilde{\nu}i))g{}_{n,m}(a)\\
=-2a(g'_{n-1,m}(a)+g'_{n-1,m-1}(a))-(1+i-2\tilde{\nu}-2i(n-2m))g{}_{n-1,m}(a)\\-(1-i-2\tilde{\nu}-2i(n-2m))g{}_{n-1,m-1}(a)
+2g{}_{n-2,m-1}(a)\\+g{}_{n-2,m}(a)+g{}_{n-2,m-2}(a)+\tilde{c}_{n,m}(a)=:R_{n,m}(a)\label{eq:iter}
\end{multline}
Note that $R_{0,0}(a)=\tilde{c}_{0,0}(a)$.

The associated homogeneous equation
\begin{multline}
(1-a^{2})g''_{n,m}(a)+(2(\tilde{\nu}-1)a+2a^{-1}+2ai(n-2m))g'_{n,m}(a)\\+(\tilde{\nu}-\tilde{\nu}^{2}+(n-2m)(n-2m+i-2\tilde{\nu}i))g{}_{n,m}(a)=0
\end{multline}
has two solutions
\[
\frac{(1-a)^{\tilde{\nu}+1+(n-2m)i}}{a},\quad \frac{(1+a)^{\tilde{\nu}+1+(n-2m)i}}{a}
\]
and their Wronskian is $2a^{-2}(\tilde{\nu}+1+(n-2m)i)(1-a^{2})^{\tilde{\nu}+(n-2m)i}$.
Therefore by (\ref{eq:iter}), $g{}_{n,m}$ can be defined recursively as
\begin{multline}
g{}_{n,m}(a)=\frac{(1+a)^{\tilde{\nu}+1+(n-2m)i}}{2a(\tilde{\nu}+1+(n-2m)i)}\int_{0}^{a}x(1+x)^{-\tilde{\nu}-1-(n-2m)i}R_{n,m}(x)\,dx\\
-\frac{(1-a)^{\tilde{\nu}+1+(n-2m)i}}{2a(\tilde{\nu}+1+(n-2m)i)}\int_{0}^{a}x(1-x)^{-\tilde{\nu}-1-(n-2m)i}R_{n,m}(x)\, dx\label{eq:gR}
\end{multline}
which implies
\begin{multline}
(ag{}_{n,m}(a))'=\frac{(1+a)^{\tilde{\nu}+(n-2m)i}}{2}\int_{0}^{a}x(1+x)^{-\tilde{\nu}-1-(n-2m)i}R_{n,m}(x)\, dx\\
+\frac{(1-a)^{\tilde{\nu}+(n-2m)i}}{2}\int_{0}^{a}x(1-x)^{-\tilde{\nu}-1-(n-2m)i}R_{n,m}(x)\, dx\label{eq:gRd}
\end{multline}
With $g{}_{n,m}$ thus defined, \eqref{eq:w0exp} gives a formal solution to \eqref{reducedw}. In order to show that \eqref{eq:w0exp} gives a true solution, it is sufficient to show that $g_{n,m}''(a)$ is continuous and
for some $C_{0}>0$ we have $$\| g_{n,m}^{(k)}\| _{\infty}:=\sup_{a\in[0,1]}|g_{n,m}^{(k)}(a)|\leqslant C_{0}^{n}$$
for $0\leqslant k\leqslant2$, since this would imply (\ref{eq:w0exp})
is convergent and twice differentiable in both $a$ and $t$ with continuous second derivatives for $0< t<t_0,0\leqslant a\leqslant 1$, as long as $\tilde{\eps}<C_{0}^{-1}$. To show that the initial conditions $q(0,t)=\partial_{a}q(0,t)=0$ are satisfied, we only need to show $g{}_{n,m}(0)=g{}_{n,m}'(0)=0$.
By differentiating \eqref{eq:w0exp} we see that
$\partial_{t}^k\partial_{a}^\ell q(a,t)=O(t^{-k}a^{2-\ell})$ for  $k\geqslant 0$ and
$ 0\leqslant \ell\leqslant 2$. In addition, to show the second estimate of \eqref{q1a3} we will prove the inequality
\begin{equation}\label{gnmd2}
|g_{n,m}''(a)|\lesssim a\tilde{C}_3^{n+1}
\end{equation} for some $\tilde{C}_3$ and $0\leqslant a\leqslant 1/2$. Note that we do not need
to show $g_{n,m}(a)=\bar{g}_{n,n-m}(a)$ since we can simply take the real
part of $q(a,t)$ in (\ref{eq:w0exp}) to get a real solution.

Since
\begin{align*}
&\left|(1+a)^{\tilde{\nu}+(n-2m)i}\int_{0}^{a}x(1+x)^{-\tilde{\nu}-1-(n-2m)i}R_{n,m}(x)dx\right|\\
&\leqslant a\| R_{n,m}\| _{\infty}(1+a)^{\tilde{\nu}}\int_{0}^{a}(1+x)^{-\tilde{\nu}-1}dx\\
&= a\tilde{\nu}^{-1}((1+a)^{\tilde{\nu}}-1)||R_{n,m}||_{\infty}\\
&\leqslant a^2(1+a)^{\tilde{\nu}-1}||R_{n,m}||_\I
\end{align*}
\begin{align*}
&\left|(1-a)^{\tilde{\nu}+(n-2m)i}\int_{0}^{a}x(1-x)^{-\tilde{\nu}-1-(n-2m)i}R_{n,m}(x)dx\right|\\
&\leqslant a\| R_{n,m}\| _{\infty}(1-a)^{\tilde{\nu}}\int_{0}^{a}(1-x)^{-\tilde{\nu}-1}dx\\
&= a\tilde{\nu}^{-1}(1-(1-a)^{\tilde{\nu}})||R_{n,m}||_{\infty}\\
&\leqslant a^2||R_{n,m}||_{\infty}
\end{align*}
and $\sqrt{2}|z|\geqslant|\Re z|+|\Im z|$ for any $z$, we have by
\eqref{eq:gRd}
\begin{align}
|g{}_{n,m}(a)| &\leqslant\frac{\sqrt{2}(2^{\tilde{\nu}-1}+2^{-1})a\|R_{n,m}\|_{\infty}}{(|n-2m|+\tilde{\nu}+1)} \nn \\
|(ag{}_{n,m}(a))'| &\leqslant\sqrt{2}(2^{\tilde{\nu}-2}+2^{-1})a^2\|R_{n,m}\|_{\infty}\label{eq:gestr} \\
|g'_{n,m}(a)| &\leqslant a^{-1}(|(ag{}_{n,m}(a))'|+|g{}_{n,m}(a)|)\nn \\
& \leqslant\sqrt{2}(2^{\tilde{\nu}-1}+2^{-1})(1+(\tilde{\nu}+1)^{-1})\|R_{n,m}\|_{\infty} \nn
\end{align}
for all $a\in[0,1]$. We let $$\hat{c}_1=\sup_{0\leqslant a \leqslant 1, 0\leqslant m \leqslant n\leqslant N}(|\tilde{c}_{n,m}(a)|,|\tilde{c}_{n,m}'(a)|)$$
By (\ref{eq:iter}) and (\ref{eq:gestr}) we
have
\begin{align}
|R_{n,m}(a)| &=|-2((ag{}_{n-1,m}(a))'+(ag{}_{n-1,m-1}(a))') \nn \\  & \quad -(-1+i-2\tilde{\nu}-2i(n-2m))g{}_{n-1,m}(a)) \nn \\
& \quad -(-1-i-2\tilde{\nu}-2i(n-2m))g{}_{n-1,m-1}(a) \nn  \\
& \quad+2g{}_{n-2,m-1}(a)+g{}_{n-2,m}(a)+g{}_{n-2,m-2}(a)+\tilde{c}_{n,m}(a)| \nn \\
& \leqslant 2(|(ag{}_{n-1,m}(a))'|+|(ag{}_{n-1,m-1}(a))'|) \nn \\ & \quad+2(\tilde{\nu}+1+|n-2m|)(|g{}_{n-1,m}(a)|+|g{}_{n-1,m-1}(a)|) \nn \\
& \quad+4\max_{0\leqslant j\leqslant2}|g{}_{n-2,m-j}(a)|+\hat{c}_1 \\
&\leqslant C_{1}\max_{1\leqslant j\leqslant2,0\leqslant k\leqslant j}\| R_{n-j,m-k}\| _{\infty}+\hat{c}_1  \label{eq:rnmest}
\end{align}
for some $C_{1}>1$. Since $|R_{0,0}(a)|=|\tilde{c}_{0,0}(a)|\leqslant \hat{c}_1$, we have
by induction
\begin{equation}
|R_{n,m}(x)|\leqslant(C_{1}+\hat{c}_1)^{n+1}\label{eq:Rest}
\end{equation}
which implies by (\ref{eq:gestr})
\begin{equation}
\| g{}_{n,m}\| _{\infty}\leqslant\frac{\tilde{C}_{1}^{n+1}}{|n-2m|+\tilde{\nu}+1},\quad \| g{}_{n,m}'\| _{\infty}\leqslant\tilde{C}_{1}^{n+1}\label{eq:gest}
\end{equation}
for some $\tilde{C}_{1}>1$. Note that by \eqref{eq:gRd} $g_{n,m}'$ is differentiable, implying $R_{n,m}$ is continuous by \eqref{eq:iter}, and thus we know $g_{n,m}''$ is continuous by differentiating \eqref{eq:gRd}. To estimate $g{}_{n,m}''$, we rewrite
(\ref{eq:gRd}) using integration by parts as
\begin{align*}
(ag{}_{n,m}(a))' &=\frac{(1+a)^{\tilde{\nu}+(n-2m)i}}{2}\int_{0}^{a}x(1+x)^{-\tilde{\nu}-1-(n-2m)i}R_{n,m}(x)\, dx \\
& \!\!\!\!+ \frac{aR_{n,m}(a)}{2(\tilde{\nu}+(n-2m)i)}
-\frac{(1-a)^{\tilde{\nu}+(n-2m)i}}{2(\tilde{\nu}+(n-2m)i)}\int_{0}^{a}(1-x)^{-\tilde{\nu}-(n-2m)i}(xR_{n,m}(x))' \, dx
\end{align*}
which implies
\begin{multline}
|2(ag{}_{n,m}(a))''|\leqslant\left|(\tilde{\nu}+(n-2m)i)(1+a)^{\tilde{\nu}-1+(n-2m)i}\int_{0}^{a}x(1+x)^{-\tilde{\nu}-1-(n-2m)i}R_{n,m}(x)dx\right|\\
+(1+a)^{-1}a|R_{n,m}(a)|+\left|(1-a)^{\tilde{\nu}-1+(n-2m)i}\int_{0}^{a}(1-x)^{-\tilde{\nu}-(n-2m)i}(xR_{n,m}(x))'dx\right|\\
\leqslant a(1+a)^{-1}(\tilde{\nu}^{-1}2^{\tilde{\nu}}(\tilde{\nu}+|n-2m|)+1)\| R_{n,m}\| _{\infty}\\
+(\tilde{\nu}-1)^{-1}(1-(1-a)^{\tilde{\nu}-1})\| (xR_{n,m}(x))'\| _{\infty}\label{eq:agad2}
\end{multline}
This together with (\ref{eq:gestr}) and (\ref{eq:Rest}) implies
\begin{align}
|(aR_{n,m}(a))'| & \leqslant 2(|(ag{}_{n-1,m}(a))''|+|(ag{}_{n-1,m-1}(a))''|) \nn \\ &  \qquad + 2(\tilde{\nu}+1+|n-2m|)(|(ag{}_{n-1,m}(a))'|+|(ag{}_{n-1,m-1}(a))'|)  \nn \\
& \qquad +4\max_{0\leqslant j\leqslant2}|(ag{}_{n-2,m-j}(a))'|+2\hat{c}_1 \nn \\
& \leqslant C_{2}\Big((\tilde{\nu}+1+|n-2m|)(C_{1}+\hat{c}_1)^{n} \nn \\
&\qquad +\max_{0\leqslant k\leqslant1}\| (xR_{n-1,m-k}(x))'\| _{\infty}\Big)+2\hat{c}_1\label{eq:aRad}
\end{align}
for some $C_{2}>1$. In particular $|(aR_{0,0}(a))'|\leqslant2\hat{c}_1$.
Thus we have by induction
\begin{equation}\label{eq:aRad2}
\| (xR_{n,m}(x))'\| _{\infty}\leqslant(\tilde{\nu}+2+n)\tilde{C}_{2}^{n+1}
\end{equation}
where $\tilde{C}_{2}=2(C_{2}+\hat{c}_1)(C_{1}+\hat{c}_1)$.
Therefore, by (\ref{eq:Rest}), (\ref{eq:agad2}), and (\ref{eq:aRad2})
\begin{equation}\label{eq:agdd}
|(ag{}_{n,m}(a))''|\leqslant a\hat{C}_2^{n+1}
\end{equation}
for some $\hat{C}_2>1$, where we used the fact that in (\ref{eq:agad2}) we have $0\leqslant1-(1-a)^{\tilde{\nu}-1}\lesssim a$.
 Now, integrating the estimate for $(ag{}_{n,m}(a))'$ in \eqref{eq:gestr} we get
\begin{align}
|g{}_{n,m}(a)| &\leqslant 3^{-1}\sqrt{2}(2^{\tilde{\nu}-2}+2^{-1})a^2\|R_{n,m}\|_{\infty}
\label{eq:gestr2}
\end{align}
which together with \eqref{eq:gestr} and (\ref{eq:Rest}) implies
\begin{align}
|g'_{n,m}(a)| &\leqslant a^{-1}(|(ag{}_{n,m}(a))'|+|g{}_{n,m}(a)|)\nn \\
& \leqslant\sqrt{2}(2^{\tilde{\nu}-1}+2^{-1})(1+3^{-1})a(C_{1}+\hat{c}_1)^{n+1}
\label{eq:gestr3}
\end{align}
By  \eqref{eq:agdd} and \eqref{eq:gestr3} we have for some $C_{0}>\tilde{C}_{1}$
\begin{equation}
|g{}_{n,m}''(a)|\leqslant a^{-1}(|(ag{}_{n,m}(a))''|+2|g{}_{n,m}'(a)|)\leqslant C_{0}^{n+1}\label{eq:gd2est}
\end{equation}
By  (\ref{eq:gest}) and (\ref{eq:gd2est}) we have $\| g_{n,m}^{(k)}(a)\| _{\infty}\leqslant C_{0}^{n}$
for $0\leqslant k\leqslant2$. Since $R_{n,m}$ is continuous (cf.~the discussion below \eqref{eq:gest}), writing  $R_{n,m}(x)=R_{n,m}(0)+o(1)$ and expanding \eqref{eq:gR} at $a=0$ we get $g_{n,m}(a)=o(a)$, implying $g_{n,m}(0)=g_{n,m}'(0)=0$.

\noindent In addition, for $q_1$ we have $\tilde{c}_{n,m}(a)=\tilde{c}_{n,m}a$. By \eqref{eq:gest} we have
$|g_{n,m}'(a)|\leqslant \tilde{C}_{1}^{n+1}$ and $|g_{n,m}(a)|\leqslant a\tilde{C}_{1}^{n+1}$ and thus by definition (cf.~\eqref{eq:iter})
 $$|R_{n,m}(a)|\leqslant aC_3^{n+1} $$
for some $C_3>0$. By \eqref{eq:iter} we have
\begin{align*}
&|(ag_{n,m}(a))''| =(1-a^2)^{-1}|2a(\nu+(n-2m)i)(ag_{n,m}(a))'\\
& +(-\tilde{\nu}-\tilde{\nu}^{2}+(n-2m)(n-2m-i-2\tilde{\nu}i))ag{}_{n,m}(a)-aR_{n,m}(a)|
\leqslant a^2\tilde{C}_3^{n+1}
\end{align*}
for some $\tilde{C}_3>0$ as long as $0\leqslant a\leqslant 1/2$, which implies \eqref{gnmd2} by direct calculation.
Therefore, $q(a,t)$ given by \eqref{eq:w0exp} is a solution to \eqref{reducedw} satisfying the stated conditions (see the discussion below \eqref{eq:gRd}).
\end{proof}

\begin{rem}
One can modify the proof of Lemma \ref{lem1} so that the results hold for $\la(t)=t^{-1-\nu}F_a( \sin(\log t), \cos(\log t))$ where $F_a(u,v)$ is analytic in $u$ and $v$ at the origin with sufficiently small derivatives. In this case, estimates of the type \eqref{eq:rnmest} remain valid.

Similarly, the results of Lemma \ref{lem1} hold for $\la(t)=t^{-1-\nu}F_b(t^{\gamma})$ where $F_b$ is analytic at the origin with sufficiently small derivatives, and ${\gamma}\in \mathbb{R}^+$. In this case, instead of \eqref{eq:w0exp} one considers
\begin{equation}
q(a,t)=\sum_{n=0}^{\infty}g_{n}(a)t^{n\gamma}
\end{equation}
and the rest of the proof is similar to that of Lemma \ref{lem1}.
\end{rem}

Using  $a=R\mu^{-1}$ we may rewrite~\eqref{v2} in the form
\EQ{\label{v2 1}
v_2(t,r):= \frac{\la(t)^{\f12}}{\mu^2(t)} \big( R \tilde q_1(a,t)+ q_2(a,t))
}
where we have set $\tilde q_1(a,t):= a^{-1} q_1(a,t)$. Note that both $\ti q_1$ and $q_2$ are~$O(a^2)$ as $a\to0$. Thus by Lemma \ref{lem1} we have the estimate
\begin{equation}\label{v2est}
\partial_t^k\partial_a^jv_2(t,at)=O\left( t^{-k}a^{-j} \frac{\la(t)^{\f12}a^{2}(1+R)}{\mu^2(t)}\right)
\end{equation}
for $k\geqslant 0$ and $0\leqslant j\leqslant 2$.
Furthermore we have
 \begin{lem}\label{lem2} For $0\leqslant r\leqslant t/2$ the estimate \eqref{v2est} holds for $k, j\geqslant 0$, or equivalently,
$$ |\partial_t^k\partial_r^\ell v_2(t,r)|\leqslant C_{k,\ell} t^{-k}r^{-\ell}
\frac{\la(t)^{\f12}a^2(1+R)}{\mu^2(t)}$$ for all $k,\ell\geqslant 0$.
 \end{lem}
\begin{proof} Note that for any differentiable function $f$ we have
$r \partial_rf(t,r)=a\partial_af(t,at)$ and $t\partial_tf(t,r)=(t\partial_t-a\partial_a) f(t,at)$. This implies
\begin{equation}\label{cov1}|t^{k}r^{\ell}\partial_t^k\partial_r^\ell v_2(t,r)|\leqslant
\tilde{C}_{k,\ell} \max_{0\leqslant m\leqslant k}t^{k-m} a^{\ell+m} |\partial_t^{k-m}\partial_a^{\ell+m} v_2(t,at)|\end{equation}
Thus it is sufficient to show that
\begin{equation}\label{pkpa0}
t^{k} a^{\ell} \partial_t^{k}\partial_a^{\ell} v_2(t,at) =O\left( \frac{\la(t)^{\f12}a^{2}(1+R)}{\mu^2(t)}\right)
\end{equation}
for all $k, \ell\geqslant 0$.

 For $\ell\leqslant 2$ this follows from \eqref{v2est}.  For $\ell>2$, we only need to show
 $$ \partial_t^{k}\partial_a^{\ell} v_2(t,at) = O\left(a^{-1}t^{-k} \frac{\la(t)^{\f12}(1+R)}{\mu^2(t)}\right)$$
By \eqref{v2 1} it is sufficient to show that
\begin{equation}\label{pkpa}
\partial_t^k\partial_a^\ell \tilde{q}_1(a,t)=O( t^{-k}) ; \ \ \partial_t^k\partial_a^\ell q_2(a,t)=O( t^{-k})
\end{equation}
since $\partial_a R=a^{-1}R$ and $\partial_a^j R=0$ for $j\geqslant 2$.
By \eqref{w1w0} we have for $j=1,2$
\EQ{\label{higherd}
 (1-a^2)\p_a^2 (aq_j(a,t))=-(2\beta_j(t)a\p_a - \beta_j^2(t) -\beta_j(t) -t \dot\beta_j(t) \big) a q_j(a,t)
\\
\qquad   + t^2a\ddot  q_j(a,t)+2(\beta_j(t)+1) t a\dot q_j(a,t) -2 at\p_a(a\dot q_j(a,t))+a^{3-j}c_j(t)
 }
Note that $\partial_t^k\beta_j(t)=O(t^{-k})$, $\partial_t^kc_j(t)=O(t^{-k})$ (cf. \eqref{eq:om024}), and by Lemma \ref{lem1} we have
\EQ{\partial_t^k\partial_a^{\ell'}(aq_j(a,t))=O( t^{-k}); \ \ 0\leqslant \ell'\leqslant 2}
Thus by differentiating \eqref{higherd} and using induction on $\ell'$ we obtain
\EQ{\label{higherd1} \partial_t^k\partial_a^{\ell'}(aq_j(a,t))=O(t^{-k}); \ \ \ell'>2}
Recall that $\partial_t^kq_j(a,t)=O(a^{4-j}t^{-k})$ by Lemma \ref{lem1}. Thus by integrating \eqref{higherd1} with $\ell'=\ell+2$ we have
\EQ{\label{intaq}a\partial_t^kq_j(a,t)=\sum_{n=5-j}^{\ell+1}c_{j,n}^{(k)}(t)a^n+\hat{q}_j^{(k)}(a,t)}
 where $c_{j,n}^{(k)}=O(t^{-k})$ and $\partial_a^m\hat{q}_j^{(k)}(a,t)=O(t^{-k}a^{\ell+2-m})$ for $ 0\leqslant m\leqslant \ell+2$. Thus by differentiating \eqref{intaq} we conclude that
 $$|\partial_t^k\partial_a^\ell \tilde{q}_1(a,t)|\leqslant | \partial_a^\ell(a^{-2}\hat{q}_1^{(k)}(a,t))|\lesssim t^{-k}$$
  $$|\partial_t^k\partial_a^\ell q_2(a,t)|\leqslant \ell!|c_{2,\ell+1}^{(k)}(t)|+| \partial_a^\ell(a^{-1}\hat{q}_2^{(k)}(a,t))|
  \lesssim t^{-k}$$
i.e., \eqref{pkpa} holds.
\end{proof}

 We set  $u_2:= u_1+v_2=u_0+v_1+v_2$. Finally, \eqref{u1 R} remains valid for~$u_2$ as well since $R\leqslant \mu(t)$. In other words, $u_0$
 gives the main shape of the profile as a  function of~$R$.
 \subsection{The error from $u_2$}\label{eu2}
 We define
 \EQ{\label{e2}
 e_2 &:= \cL_{quintic} (u_2) = \cL_{quintic} (u_1+ v_2) \\
 &=  \cL_{quintic} (u_1)  + u_1^5 - (u_1+v_2)^5 +(  \p_{tt} -\p_{rr} - \f{2}{r}\p_r )v_2\\
 &= e_1 - e_1^0 - 5u_1^4 v_2 - 10 u_1^3 v_2^2 - 10 u_1^2 v_2^3 - 5 u_1 v_2^4 - v_2^5
 }
 We determine $t^2 \la(t)^{-\f12} e_2$. First, from~\eqref{ediff}
 \EQ{ \label{ediff 2}
&t^2 \la^{-\f12}(t) ( e_1- e_1^0)(t,r) \\
& = \mu^{-2}(t) \Big[ \frac{\log R}{R} \Phi_1(a,b,\rho\log\rho,\rho) + \frac{1}{R} \Phi_2(a,b,\rho\log\rho,\rho)\Big]
}
for $R\ge1$. For $|R|<1$ we read off from~\eqref{e1*} and~\eqref{e10} that
 \EQ{ \label{ediff 3}
&t^2 \la^{-\f12}(t) ( e_1- e_1^0)(t,r)  = O(\mu^{-2}(t))
}
This holds uniformly for small times, and $t\p_t$ and $r\p_r$ can be applied any number of times without changing this asymptotic behavior as $R\to0$.

Next, for large~$R$,  by \eqref{u1 R} and \eqref{v2est} we have

\[
t^2 \la^{-\f12}(t) u_1^4 v_2 =O(a^2R^{-3}) = O(R^{-1} \mu^{-2}(t))
\]
The final nonlinear term contributes
\begin{align*}
t^2 \la^{-\f12}(t)  v_2^5 &= \mu^{-8}(t) O(R^5) = \mu^{-2}(t) R^{-1} O(\mu^{-6}(t) R^6) = O(R^{-1} \mu^{-2}(t))
\end{align*}
Thus
\EQ {\label{nonli1}
t^2 \la^{-\f12}(t) u_1^k v_2^{5-k} = O(R^{-1} \mu^{-2}(t))\qquad (R\geqslant 1,0\leqslant k\leqslant 4)
}

  For small~$R$ we have $u_1=\la^{\f12}(t)O(1)$ by \eqref{u1 R} and $v_2=\la^{\f12}(t)  \mu^{-2}(t)O(a^2)$ by \eqref{v2est}. Thus (recall that $a=R\mu^{-1}(t)$)
 $$t^2 \la^{-\f12}(t) u_1^4 v_2 =O(a^2)=O(\mu^{-2}(t))$$
  $$t^2 \la^{-\f12}(t) u_1 v_2^4 =O (a^8\mu^{-6}(t))=O(\mu^{-2}(t))$$
and we have
  \EQ {\label{nonli2}
t^2\la^{-\f12}(t) u_1^k v_2^{5-k} =O(\mu(t)^{-2}) \qquad  (R\leqslant 1,0\leqslant k\leqslant 4)
}
By the preceding we gain a factor $\mu^{-2}$
for all~$R$, and the decay is at least $\frac{\log R}{R} $ as~$R\to\I$.

Finally by \eqref{u1 R} and Lemma \ref{lem2} we see that the estimates \eqref{nonli1} and \eqref{nonli2} remain valid after one applies $t\partial_t$ or $r\partial_r$ any number of times if $0\leqslant r\leqslant t/2$. Similarly by \eqref{u1 R} and \eqref{v2est} we see that \eqref{nonli1} and \eqref{nonli2} remain valid after one applies $t\partial_t$ and $r\partial_r$ no more than twice, if $0\leqslant r\leqslant t$.

\subsection{Proof of Proposition \ref{prop1}} \label{prop1proof}
(i) (ii) Smoothness of $u_2$ follows from \eqref{u0def}, \eqref{v1w1} (where $f_j$ satisfies \eqref{fasymp}, which can be differentiated), and Lemma \ref{lem1}, which imply $u_0,v_1,v_2$ are all in $C^2(\{0<t<t_0, 0\leqslant r\leqslant t\})$.

To show \eqref{u2bd}, it is sufficient to show that $v_{1,2}=\sqrt{\la(t)}\,\mu^{-2}(t)O(R)$ for both small and large $R$. This follows from \eqref{fasymp}  and \eqref{v1w1} for $v_1$, and \eqref{v2est} for $v_2$. Similarly \eqref{u2der} follows from the fact that $\p_t^k\p_r^jv_{1,2}=t^{-k}r^{-j}\sqrt{\la(t)}\,\mu^{-2}(t)O(R)$ by \eqref{fasymp} (which is clearly differentiable in $R$), \eqref{v1w1} and Lemma \ref{lem2}. For $k+j\leqslant 2$ and $0\leqslant r\leqslant t$ we use \eqref{v2est} instead of Lemma \ref{lem2}.

Finally \eqref{e2bd} and \eqref{e2der} follow from \eqref{e2}, where the different parts are estimated in \eqref{ediff 2}, \eqref{ediff 3}, \eqref{nonli1}, and \eqref{nonli2}, which remain valid after one applies $t\partial_t$ or $r\partial_r$ any number of times if $0\leqslant r\leqslant t/2$, or if they are  applied no more than twice and $0\leqslant r\leqslant t$.

(iii) We let$$\hat{u}_2(t,r)=\begin{cases} u_2(t,r), & \mbox{if } 0<t<t_0, 0\leqslant r\leqslant t
 \\ u_2(t,t)+(r-t) u_{2}^{(0,1)}(t,t)\\+\frac{1}{2} (r-t)^2u_{2}^{(0,2)}(t,t), & \mbox{if } 0<t<t_0, t< r\leqslant (1+2b_1)t \end{cases}$$
where $b_1>0$ and $u_{2}^{(n,m)}(t,r):=\p_r^m\p_t^n u_2(t,r)$. Clearly $\hat{u}_2$ is $C^2$ in $r$ for fixed $t$.

By direct calculation using \eqref{u2der} we have for $0\leqslant k+j\leqslant 2$
$$u_{2}^{(k,j)}(t,r)=O(r^{-j-1}t^{-k}\la^{-1/2}(t))$$
Thus for $t< r\leqslant (1+2b_1)t$ we have
\begin{equation}\label{uhat1}
\hat{u}_2(t,r)=O(t^{-1}\la^{-1/2}(t)) ; \ \ \p_r\hat{u}_2(t,r)=O(t^{-2}\la^{-1/2}(t))
\end{equation}
 $$\p_t\hat{u}_2(t,r)=O\left(\max_{0\leqslant m\leqslant 1, 0\leqslant n\leqslant 1}\left|(r-t)^n u_{2}^{(1-m,n+m)}(t,t)\right|+\left|(r-t)^2\p_tu_{2}^{(0,2)}(t,t)\right|
 \right)$$
where the first term is clearly of order $O(t^{-2}\la^{-1/2}(t))$. To estimate the second term, recall that $u_2=u_1+v_2$ where by direct calculation
$u_1^{(k,j)}(t,r)=O(r^{-j-1}t^{-k}\la^{-1/2}(t))$ for all $k,j\geqslant 0$. Thus
$$(r-t)^2\p_tu_{1}^{(0,2)}(t,t)=O(t^{-2}\la^{-1/2}(t))$$
Since $a=r/t$, we have
$$v_{2}^{(0,2)}(t,t)=t^{-2}\partial_a^{2} v_2(t,t)$$
This and \eqref{v2est} imply

$$\partial_t v_{2}^{(0,2)}(t,t)=O(t^{-4}\la^{-1/2}(t))$$
Therefore
\begin{equation}\label{uhat2}\p_t\hat{u}_2(t,r)=O(t^{-2}\la^{-1/2}(t))\end{equation}

Now we let $B_1$ be a smooth bump function satisfying $$B_1(x)=
\begin{cases} 1, & \mbox{if } x<1
 \\ 0, & \mbox{if } x>1+b_1 \end{cases}$$
and we let $u_2(t,r)=\hat{u}_2(t,r)B_1(r/t)$ for $r>t$. Clearly $u_2$ is $C^2$ in $r$ for fixed $t$. By direct calculation using \eqref{uhat1} and \eqref{uhat2} we have
$$u_2(t,r)=O(t^{-1}\la^{-1/2}(t)); \ \p_ru_2(t,r)=O(t^{-2}\la^{-1/2}(t)); \ \p_tu_2(t,r)=O(t^{-2}\la^{-1/2}(t))$$
Therefore $b_1$ can be chosen small enough to ensure
 $$\int_{t\leqslant r}u_2^6\,dx\lesssim  b_1 t^{-3}\la^{-3}(t)< \delta/3$$
 $$\int_{t\leqslant r}(\p_ru_2)^2\, dx\lesssim b_1 t^{-1}\la^{-1}(t)< \delta/3$$
 $$\int_{t\leqslant r}(\p_tu_2)^2\, dx\lesssim b_1 t^{-1}\la^{-1}(t)< \delta/3$$
Thus their sum is less than $\delta$.

\section{Construction of an exact solution}

Our aim next is to construct an energy class solution of \eqref{u5} of the form
\[
u = u_2 + \eps
\]
in the backward light cone $r\leq t$, $0<t<t_0$.
One immediately infers the equation
\begin{equation}\label{eqn:eps}
\Box \eps + 5u_0^4\eps = 5(u_0^4 - u_2^4)\eps - N(u_2,\eps) - e_2,
\end{equation}
where we have
\[
N(u_2,\eps) = 10 u_2^3\eps^2 + 10u_2^2\eps^3 + 5u_2\eps^4 + \eps^5;
\]
also, we shall denote by $e_2$ an extension of $e_2$ in the preceding section beyond the light cone satisfying the same asymptotics as in Proposition~\ref{prop1}.
Proceeding exactly as in \cite{DonKri12}, we pass to the variables
\[
\tau = \int_{t_0}^t\lambda(s)\,ds,\,R = \lambda(t)r,\,v(\tau, R) = R\eps(t(\tau), r(\tau, R))
\]
Writing \footnote{We warn the reader that the symbols $\beta$ and $\kappa$ have a different
meaning here than in Section \ref{sec:approx}.} $\kappa(\tau): = \lambda(t(\tau))$, as well as
$\beta(\tau): = \frac{\kappa'(\tau)}{\kappa(\tau)}$, and
\[
\mathcal{D}: = \partial_{\tau} + \beta(\tau)(R\partial_R - 1),
\]
we get
\begin{equation}\label{eq:veqn1}
[\mathcal{D}^2 + \beta(\tau)\mathcal{D} + \mathcal{L}]v = \kappa^{-2}(\tau)\big[5(u_0^4 - u_2^4)v + RN(u_2, \frac{v}{R}) + Re_2\big]
\end{equation}
where $\mathcal{L}: = -\partial_R^2 - 5W^4(R)$, and we interpret $u_2, u_0, e_2$ as functions of $\tau, R$. In fact, since it suffices to solve this problem in a dilate of the light cone, we replace it by
\begin{equation}\label{eq:veqn2}
[\mathcal{D}^2 + \beta(\tau)\mathcal{D} + \mathcal{L}]v = \kappa^{-2}(\tau)\tilde{\chi}(\frac{R}{\nu\tau})\big[5(u_0^4 - u_2^4)v + RN(u_2, \frac{v}{R}) + Re_2\big]
\end{equation}
for some smooth cutoff $\tilde{\chi}\in C_0^\infty(\R_+)$ with $\tilde{\chi}|_{r\leq 1} = 1$. In fact, the main work consists in solving the linear in-homogeneous problem
\[
[\mathcal{D}^2 + \beta(\tau)\mathcal{D} + \mathcal{L}]v = f,
\]
where $f$ satisfies bounds like $ \kappa^{-2}(\tau)\tilde{\chi}(\frac{R}{\nu\tau})Re_2$. Our approach below is a general framework to solve such problems, applicable to much wider classes of scaling factors $\lambda(t)$.

\subsection{The distorted Fourier transformation}

Here we freely borrow facts from \cite{KST} as well as \cite{DonKri12}, \cite{KrSch}. We state
\begin{thm}[Spectral theory for $\mathcal{L}$]\hfill
\label{thm:spec}
\begin{itemize}
\item The Schr\"odinger operator $\mathcal{L}$ is self-adjoint on $L^2(0,\infty)$ with domain
\begin{align*}
\mathrm{dom}(\mathcal{L})=\{&f\in L^2(0,\infty): f,f' \in \mathrm{AC}[0,R]\:\forall R>0, \\
&f(0)=0, \mathcal{L} f \in L^2(0,\infty)\}
\end{align*}
and its spectrum is given by $\sigma(\mathcal{L})=\{\xi_d\}\cup [0,\infty)$ where $\xi_d<0$.
The continuous part of the spectrum is absolutely continuous and
the eigenfunction $\phi(R,\xi_d)$
associated to the eigenvalue $\xi_d$ is smooth
and decays exponentially as $R\to\infty$.

\item The spectral measure $\mu$ is of the form
\[ d\mu(\xi)=\frac{\delta_{\xi_d}(\xi)}{\|\phi(\cdot,\xi_d)\|_{L^2(0,\infty)}^2}+\rho(\xi)d\xi \]
where $\delta_{\xi_d}$ denotes the Dirac measure centered at $\xi_d$ and
the function $\rho$ satisfies $\rho(\xi)=0$ for $\xi<0$ as well as
\footnote{The conclusion for the asymptotics near $\xi = 0$ is
not optimal and one can replace $O(\xi^\frac15)$ by $O(\xi^{\frac12})$, but we will not need this.}
\begin{align*}
\rho(\xi)&=\tfrac{1}{3\pi}\xi^{-\frac12}[1+O(\xi^\frac15)]\quad 0<\xi\leq 1 \\
\rho(\xi)&=\tfrac{1}{\pi}\xi^\frac12[1+O(\xi^{-\frac12})]\quad \xi\geq 1
\end{align*}
where the $O$-terms behave like symbols under differentiation.

\item There exists a unitary operator $\mathcal{U}: L^2(0,\infty)\to L^2(\sigma(\mathcal{L}),d\mu)$ which
diagonalizes $\mathcal{L}$, i.e., $\mathcal{U} \mathcal{L} f(\xi)=\xi \mathcal{U} f(\xi)$ for all $f\in \mathrm{dom}(\mathcal{L})$.
The operator $\mathcal{U}$ is given explicitly by
\[ \mathcal{U} f(\xi)=\lim_{b\to\infty}\int_0^b \phi(R,\xi)f(R)dR \]
where the limit is understood in $L^2(\sigma(\mathcal{L}),d\mu)$.
The function $\phi(\cdot,\xi)$ is smooth and (formally) satisfies
$\mathcal{L}\phi(\cdot,\xi)=\xi \phi(\cdot,\xi)$ as well as $\phi(0,\xi)=0$, $\phi'(0,\xi)=1$.

\item
For $0<\xi\lesssim 1$ we have the asymptotics
\begin{align*}
\phi(R,\xi)&=\phi_0(R)[1+O(\langle R \rangle^2 \xi)]\quad 0\leq R \leq \xi^{-\frac12} \\
\phi(R,\xi)&=\tfrac{\sqrt{3}}{2}e^{i\sqrt{\xi}R}[1+O(\xi^\frac15)+O(\langle R\rangle^{-3}\xi^{-\frac12})] \\
&\quad+\tfrac{\sqrt{3}}{2}e^{-i\sqrt{\xi}R}[1+O(\xi^\frac15)+O(\langle R\rangle^{-3}\xi^{-\frac12})]\quad R\geq \xi^{-\frac16}
\end{align*}
where all $O$-terms behave like symbols under differentiation and
\[ \phi_0(R):=\frac{R(1-\frac13 R^2)}{(1+\frac13 R^2)^{3/2}}. \]
In the case $\xi \gtrsim 1$ we have
\begin{align*}
\phi(R,\xi)&=\tfrac{1}{2i}\xi^{-\frac12}e^{i\sqrt{\xi}R}[1+O(\xi^{-\frac12})+O(\langle R\rangle^{-3}\xi^{-\frac12})] \\
&\quad-\tfrac{1}{2i}\xi^{-\frac12}e^{-i\sqrt{\xi}R}[1+O(\xi^{-\frac12})+O(\langle R\rangle^{-3}\xi^{-\frac12})]
\end{align*}
for all $R\geq 0$ with symbol behavior of all $O$-terms.

\item The inverse map $\mathcal{U}^{-1}: L^2(\sigma(\mathcal{L}),d\mu)\to L^2(0,\infty)$ is given by
\[ \mathcal{U}^{-1}f(R)=\frac{\phi(R,\xi_d)}{\|\phi(\cdot,\xi_d)\|_{L^2(0,\infty)}^2}f(\xi_d)
+\lim_{b\to\infty}\int_0^b \phi(R,\xi)f(\xi)\rho(\xi)d\xi \]
where the limit is understood in $L^2(0,\infty)$.
\end{itemize}
\end{thm}

For the following it turns out to be more convenient to use a vector-valued version of $\mathcal{U}$
which we denote by $\mathcal{F}$ and call the ``distorted Fourier transform''.
Thus, we identify $L^2(\sigma(\mathcal{L}),d\mu)$ with $\C \times L^2_\rho(0,\infty)$ and define the mapping
$\mathcal{F}: L^2(0,\infty)\to \C \times L^2_\rho(0,\infty)$ by
\[ \mathcal{F} f=\left (\begin{array}{c}\mathcal{U}f(\xi_d) \\
\mathcal{U} f|_{[0,\infty)} \end{array} \right ). \]
According to Theorem \ref{thm:spec}, the inverse map $\mathcal{F}^{-1}: \C \times L^2_\rho(0,\infty) \to L^2(0,\infty)$ is then given by
\[ \mathcal{F}^{-1} \left (\begin{array}{c} a \\ f \end{array} \right )=\frac{\phi(R,\xi_d)}{\|\phi(\cdot,\xi_d)\|_{L^2(0,\infty)}^2}a
+\lim_{b\to\infty}\int_0^b \phi(R,\xi)f(\xi)\rho(\xi)d\xi. \]

From now on we shall write
\begin{align*}
v(\tau, R) = x_d(\tau)\phi_d(R) + \int_0^\infty x(\tau, \xi)\phi(R, \xi)\rho(\xi)\,d\xi
\end{align*}
where the functions $x(\tau, \xi)$ are the (distorted) Fourier coefficients associated with $v(\tau, \cdot)$. We write
\[
\underline{x}(\tau, \xi): = \left(\begin{array}{c}x_d(\tau)\\ x(\tau, \xi)\end{array}\right) = \mathcal{F}(v)(\tau, \xi),\,\underline{\xi}: = \left(\begin{array}{c}\xi_d\\ \xi\end{array}\right)
\]
Then precisely as in \cite{KrSch}, we obtain the relation
\begin{equation}\label{eq:transport}
\big(\hat{\mathcal{D}}^2 + \beta(\tau)\hat{\mathcal{D}} +
\underline{\xi}\big)\underline{x}(\tau, \xi) = \mathcal{R}(\tau, \underline{x}) + f(\tau, \underline{\xi}),
\end{equation}
where we have
\begin{equation}\label{eq:Rterms}
\mathcal{R}(\tau, \underline{x})(\xi) =
\Big(-4\beta(\tau)\mathcal{K}\hat{\mathcal{D}}\underline{x}
- \beta^2(\tau)(\mathcal{K}^2 + [\mathcal{A}, \mathcal{K}] + \mathcal{K} +
\beta' \beta^{-2}\mathcal{K})\underline{x}\Big)(\xi)
\end{equation}
 with $\beta(\tau) = \frac{\dot{\kappa}(\tau)}{\kappa(\tau)}$, and
 \begin{equation}\label{eq:fterms}
 f(\tau, \xi) = \mathcal{F}\big( \kappa^{-2}(\tau)\big[5(u_{2k-1}^4 - u_0^4)v + RN(u_{2k-1}, v) + R e_{2}\big]\big)\big(\xi\big)
 \end{equation}
as well as the operator
 \[
 \hat{\mathcal{D}} =
 \partial_{\tau} + \beta(\tau)\mathcal{A},\quad \mathcal{A} = \left(\begin{array}{cc}0&0\\0&\mathcal{A}_c\end{array}\right)
 \]
 with
 \[
 \mathcal{A}_c = -2\xi\partial_{\xi} - \Big (\frac{5}{2}  + \frac{\rho'(\xi)\xi}{\rho(\xi)} \Big)
 \]
 Finally, we observe that the ``transference operator'' $\mathcal{K}$ is given by the following type of expression
 \begin{equation}\label{eq:Kstructure}
 \mathcal{K} = \left(\begin{array}{cc}\mathcal{K}_{dd}&\mathcal{K}_{dc}\\
 \mathcal{K}_{cd}&\mathcal{K}_{cc}\end{array}\right)
 \end{equation}
 where the matrix elements are certain non-local Hilbert type operators.
 Then we use the key observation, already made in \cite{KrSch}, that the abstract problem \eqref{eq:transport} with $\mathcal{R}(\tau, \underline{x}) = 0$ can be solved {\it{explicitly}} for the continuous part $x(\tau, \xi)$.  In fact, we have the relation
 \begin{equation}\label{eq:xpara1}
 x(\tau, \xi) = \int_{\tau}^\infty H_c(\tau, \sigma, \xi)f\big(\sigma, \frac{\kappa^{2}(\tau)}{\kappa^{2}(\sigma)}\xi\big)\,d\sigma
 \end{equation}
 with
  \begin{equation}\label{eq:xpara2}
H_c(\tau, \sigma, \xi) = \xi^{-\frac{1}{2}}\frac{\kappa^{\frac{3}{2}}(\tau)}{\kappa^{\frac{3}{2}}(\sigma)}\frac{\rho^{\frac{1}{2}}(\frac{\kappa^{2}(\tau)}{\kappa^{2}(\sigma)}\xi)}{\rho^{\frac{1}{2}}(\xi)}\sin\Big[\kappa(\tau)\xi^{\frac{1}{2}}\int_{\tau}^{\sigma}\kappa^{-1}(u)\,du\Big]
\end{equation}
Furthermore, letting (as in \cite{DonKri12})
\[
\hat{\mathcal{D}}_{c}: = \partial_{\tau} + \beta(\tau)\mathcal{A}_c,
\]
one computes from the above parametrix representation that
\begin{equation}\label{eq:xpara3}
\hat{\mathcal{D}}_{c}x(\tau, \xi) = \int_{\tau}^\infty \hat{H}_c(\tau, \sigma, \xi)f\big(\sigma, \frac{\kappa^{2}(\tau)}{\kappa^{2}(\sigma)}\xi\big)\,d\sigma
\end{equation}
with
\begin{equation}\label{eq:xpara4}
 \hat{H}_c(\tau, \sigma, \xi) = \frac{\kappa^{\frac{3}{2}}(\tau)}{\kappa^{\frac{3}{2}}(\sigma)}\frac{\rho^{\frac{1}{2}}(\frac{\kappa^{2}(\tau)}{\kappa^{2}(\sigma)}\xi)}{\rho^{\frac{1}{2}}(\xi)}\cos\Big[\kappa(\tau)\xi^{\frac{1}{2}}\int_{\tau}^{\sigma}\kappa^{-1}(u)\,du\Big]
\end{equation}
We can immediately formulate the following
\begin{lem}\label{lem:parbounds}
Denoting $\omega_{\nu}(\tau): = \tau^{1+\frac{1}{\nu}}$, and letting $\kappa(\tau) = \lambda(t(\tau))$ as in the preceding, we have the kernel bounds
\[
|H_c(\tau, \sigma, \xi)|\lesssim \min\{\omega_{\nu}(\frac{\tau}{\sigma})\xi^{-\frac{1}{2}},\,\nu\omega_{\nu}(\frac{\tau}{\sigma})\sigma\}
\]
\[
|\hat{H}_c(\tau, \sigma, \xi)|\lesssim \omega_{\nu}(\frac{\tau}{\sigma})
\]
\end{lem}
Indeed, this is a simple consequence of the fact that
\[
\kappa(\tau)\sim \tau^{1+\frac{1}{\nu}}
\]
For the discrete part $x_d(\tau)$ of the solution of \eqref{eq:transport}
 with $\mathcal{R}(\tau, \underline{x})  = 0$, we obtain the implicit equation
\begin{equation}\begin{split}\label{eq:x_d}
&x_d(\tau) = \int_{\tau}^\infty H_d(\tau, \sigma)\tilde{f}_d(\sigma)\,d\sigma,\; \; H_d(\tau, \sigma) = -\frac{1}{2}|\xi_d|^{-\frac{1}{2}}e^{-|\xi_d|^{\frac{1}{2}}|\tau-\sigma|}\\
&\tilde{f}_d(\sigma) = f_d(\sigma) -  \beta_{\nu}(\sigma)\partial_{\sigma}x_d(\sigma),
\end{split}\end{equation}
\\
In order to solve the problem \eqref{eq:transport} via a fixed point argument, we shall utilize the functional framework developed in \cite{DonKri12}:

\begin{defn}\label{def:spaces}
For the continuous spectral part $x(\tau, \xi)$, we shall use the following norms:
\begin{align*}
\|f\|_{X}: = \big\|(|\cdot|\langle\cdot\rangle^{-1})^{\frac{1}{2}-\delta}f\big\|_{L^p(0, \infty)} + \big\| \, |\cdot|^{\frac{1}{2}}\langle\cdot\rangle^{\frac{1}{8}}f\big\|_{L^2_{\rho}(0,\infty)}
\end{align*}
\begin{align*}
\|f\|_{Y}: = \big\|f\big\|_{L^p(0, \infty)} + \big\|\langle\cdot\rangle^{\frac{1}{8}}f\big\|_{L^2_{\rho}(0,\infty)}
\end{align*}
as well as
\[
\|u\|_{\mathcal{X}^{\beta}}: = \sup_{\tau>\tau_0}\tau^{\beta}\|u(\tau, \cdot)\|_{X},\,\|u\|_{\mathcal{Y}^{\beta}}: = \sup_{\tau>\tau_0}\tau^{\beta}\|u(\tau, \cdot)\|_{Y}
\]
Then for the vector valued function $\underline{x}(\tau, \xi)$, we put
\begin{align*}
\big\|\underline{x}\big\|_{\mathcal{X}^{\alpha, \beta}}: = \sup_{\tau>\tau_0}\tau^{\alpha}|x_d(\tau)| + \|x(\tau, \cdot)\|_{\mathcal{X}^{\beta}}
\end{align*}
\begin{align*}
\big\|\underline{x}\big\|_{\mathcal{Y}^{\alpha, \beta}}: = \sup_{\tau>\tau_0}\tau^{\alpha}|x_d(\tau)| + \|x(\tau, \cdot)\|_{\mathcal{Y}^{\beta}}
\end{align*}
We remark that in the following $\delta>0$ is assumed to be small and $p>1$ is assumed
to be large, depending on $\delta$.
\end{defn}
To proceed, we first need to study the linear inhomogeneous problem
\begin{equation}\label{eq:lininhom}
\big(\hat{\mathcal{D}}^2 + \beta(\tau)\hat{\mathcal{D}} + \underline{\xi}\big)\underline{x}(\tau, \xi) =  \underline{f}(\tau, \underline{\xi})
\end{equation}
To prepare for this task, we have
\begin{lem}
\label{lem:estB}
Let $\kappa(\tau) = \lambda(t(\tau))$ as in the preceding. Let $a,b,\gamma \in \R$, $q\in (1,\infty)$,
and
\[ \alpha > 1+2\left (\tfrac{1}{q}+\max(|a|,|a+b|)\right )(1+\tfrac{1}{\nu})-\gamma. \]
Suppose further
that the operator $\mathcal{B}$ is given by
\[ \mathcal{B} x(\tau,\xi)=\int_\tau^\infty B(\tau,\sigma,\xi)
x(\sigma,\omega(\tau, \sigma)^2\xi)d\sigma, \]
where $\omega(\tau, \sigma):=\kappa(\tau)\kappa^{-1}(\sigma)$,
and the kernel $B$ satisfies
\[ |B(\tau,\sigma,\xi)|\lesssim \sigma^{-\gamma} \]
for all $0<\tau_0\leq \tau \leq \sigma$, $\xi\geq 0$.
Then we have the bound
\[ \|\mathcal{B} x(\tau,\cdot)|\cdot|^a\langle \cdot\rangle^b\|_{L^q(0,\infty)}
\lesssim \tau^{-\alpha-\gamma+1}\sup_{\sigma>\tau}\sigma^\alpha\|x(\sigma,\cdot)
|\cdot|^a \langle \cdot \rangle^b\|_{L^q(0,\infty)}. \]
\end{lem}

\begin{proof}
First, we consider the case $a=b=0$.
By H\"older's inequality we obtain
\begin{equation}
\label{eq:Bx}
|\mathcal{B} x(\tau,\xi)|\lesssim \left (\int_\tau^\infty \sigma^{-1-\epsilon}d\sigma \right )^{1/q'}
\left (\int_\tau^\infty |\sigma^{\frac{1}{q'}(1+\epsilon)-\gamma}x(\sigma,\omega(\tau, \sigma)^2\xi)|^q d\sigma
\right )^{1/q}
\end{equation}
for any $\epsilon>0$.
This implies
\begin{align*}
\|\mathcal{B} x(\tau,\cdot)\|_{L^q}&\lesssim \tau^{-\frac{\epsilon}{q'}}\left (
\int_\tau^\infty \sigma^{q(\frac{1}{q'}(1+\epsilon)-\gamma-\alpha)}\|\sigma^\alpha x(\sigma,
\omega(\tau, \sigma)^2 \cdot)\|_{L^q}^q \right )^{1/q} \\
&\lesssim \tau^{-\frac{\epsilon}{q'}}\left [\sup_{\sigma>\tau}\sigma^\alpha
\|x(\sigma,\cdot)\|_{L^q}\right ]
\left (\int_\tau^\infty \sigma^{q(\frac{1}{q'}(1+\epsilon)-\gamma-\alpha)}\omega_{\nu}(\tfrac{\tau}{\sigma})^{-2}
d\sigma \right )^{1/q} \\
&\lesssim \tau^{-\alpha-\gamma+1}\sup_{\sigma>\tau}\sigma^\alpha
\|x(\sigma,\cdot)\|_{L^q}
\end{align*}
provided $q(\frac{1}{q'}-\gamma-\alpha)+2(1+\frac{1}{\nu})<-1$ and $\epsilon>0$ is chosen
sufficiently small.
The case for general $a,b$ follows immediately by noting that
\[ \|x(\sigma,\omega(\tau, \sigma)^2\cdot)|\cdot|^a \langle \cdot \rangle^b\|_{L^q}^q
\lesssim \omega_{\nu}(\tfrac{\tau}{\sigma})^{-2-2q\max(|a|,|a+b|)}
\|x(\sigma,\cdot)|\cdot|^a \langle \cdot \rangle^b\|_{L^q}^q \]
which entails the integrability condition
\[ q(\tfrac{1}{q'}-\gamma-\alpha)+\left (2+2q\max(|a|,|a+b|)\right )(1+\tfrac{1}{\nu})<-1. \]
\end{proof}

We can now solve \eqref{eq:lininhom} by the following

\begin{lem}
\label{lem:estH}
Let $\alpha_d \in \R$ and $\alpha_c>1+\frac34(1+\frac{1}{\nu})$. Then given $\underline{f}\in \mathcal{Y}^{\alpha_d,\alpha_c}$, there exists a solution $\underline{x}\in \mathcal{X}^{\alpha_d,\alpha_c-1-2\delta}$ for \eqref{eq:lininhom}. Denoting this solution by
\[
\underline{x} =: \left(\begin{array}{c} \mathcal{H}_d f_d \\ \mathcal{H}_{c}f\end{array}\right) =:\mathcal{H}\underline{f},
\]
we have the estimates
\begin{align*}
\|\mathcal{H} \underline{x}\|_{\mathcal{X}^{\alpha_d,\alpha_c-1-2\delta}}&\lesssim \nu^{2\delta}\|\underline{x}\|_{\mathcal{Y}^{\alpha_d,\alpha_c}} \\
\|\hat{\mathcal{D}}\mathcal{H} \underline{x}\|_{\mathcal{Y}^{\alpha_d,\alpha_c-1}}&\lesssim \|\underline{x}\|_{\mathcal{Y}^{\alpha_d,\alpha_c}}
\end{align*}
where $\delta>0$ is the parameter in Definition \ref{def:spaces}.
\end{lem}
\begin{proof} We can explicitly define the continuous spectral part $x(\tau, \xi)$ via \eqref{eq:xpara1}, \eqref{eq:xpara2}, and the discrete part $x_d(\tau)$ implicitly via \eqref{eq:x_d}.
Combining \eqref{eq:xpara3}, \eqref{eq:xpara4}, Lemma~\ref{lem:parbounds} as well as Lemma \ref{lem:estB}
we get $\|\hat{\mathcal{D}}_c \mathcal{H}_c x\|_{\mathcal{Y}^{\alpha_c-1}}\lesssim \|x\|_{\mathcal{Y}^{\alpha_c}}$.
On the other hand, using
\[ |H_c(\tau,\sigma,\xi)|\lesssim \nu^{2\delta}\omega_{\nu}(\tfrac{\tau}{\sigma})\sigma^{2\delta}
\xi^{-\frac12+\delta} \]
which follows from Lemma~\ref{lem:parbounds} by interpolation, as well as Lemma~\ref{lem:estB}, we have
\[ \| \, |\cdot|^{\frac12-\delta}\mathcal{H}_c x(\tau,\cdot)\|_{L^p}\lesssim \nu^{2\delta}
\tau^{-\alpha_c+1+2\delta}\|x\|_{\mathcal{Y}^{\alpha_c}}. \]
Further, Lemma~\ref{lem:estB} gives
\[ \| \, |\cdot|^\frac12 \langle \cdot\rangle^\frac18 \mathcal{H}_c x(\tau,\cdot)\|_{L^2_\rho}\lesssim
\| \, |\cdot|^\frac12 \mathcal{H}_c x(\tau,\cdot)\|_Y\lesssim \tau^{-\alpha_c+1}\|x\|_{\mathcal{Y}^{\alpha_c}}. \]
This completes the desired bounds for the continuous part $x(\tau, \cdot)$.
To control the discrete part, we observe that, see \eqref{eq:x_d}
\[
\sup_{\tau>\tau_0}\tau^{\alpha_d}\big|\int_{\tau}^\infty H_d(\tau, \sigma)\,f(\sigma)\,d\sigma\big|\lesssim \sup_{\tau>\tau_0}\tau^{\alpha_d}|f(\tau)|,
\]
\[
\sup_{\tau>\tau_0}\tau^{\alpha_d}\big|\partial_{\tau}\int_{\tau}^\infty H_d(\tau, \sigma)\,f(\sigma)\,d\sigma\big|\lesssim \sup_{\tau>\tau_0}\tau^{\alpha_d}|f(\tau)|,
\]
In light of the fact that
\[
\beta_{\nu}(\tau)\sim \frac{1}{\tau},
\]
the implicit equation
\[
x_d(\tau) = \int_{\tau}^\infty H_d(\tau, \sigma)\big(f_d(\sigma) -  \beta_{\nu}(\sigma)\partial_{\sigma}x_d(\sigma)\big)\,d\sigma
\]
is then solved via straightforward iteration provided $\tau>\tau_0$ with $\tau_0$ sufficiently large, and the limit satisfies
\[
\sup_{\tau>\tau_0}\tau^{\alpha_d}|x_d(\tau)|\lesssim \sup_{\tau>\tau_0}\tau^{\alpha_d}|f_d(\tau)|
\]
This completes the proof of the lemma.
\end{proof}

\subsection{Solving the main equation}
\label{subsec:solmain}
Abstractly speaking, Eq.~\eqref{eq:transport} is of the form
\begin{equation}
\label{eq:abstract}
Lx=x_0+Ax+F(x)
\end{equation}
where $x_0$ is a given element in a Banach space $X$, $A$ is a bounded linear operator on $X$,
and $F$ is a nonlinear mapping from $X$ to $X$.
Furthermore, the operator $L$ is linear and invertible with bounded inverse $H$, in light of Lemma~\ref{lem:estH}.
The goal is to find a solution $x\in X$.
Compared to Eq.~\eqref{eq:transport}, this is a slightly simplified model case but it captures the essentials.
By applying $H$, one rewrites Eq.~\eqref{eq:abstract} as
\begin{equation}
\label{eq:abstract2}
x=Hx_0+HAx+HF(x).
\end{equation}
The point is to find a method to solve Eq.~\eqref{eq:abstract2}, even if the operator norm
of $HA$ is not small, i.e., if one cannot apply the Banach
fixed point theorem directly.
The idea is to perform an iteration procedure.
This means that one first proves the existence of $(1-HA)^{-1}$ which amounts to showing the norm-convergence of
the Neumann series
\[ \sum_{n=0}^\infty (HA)^n. \]
Thus, one has to consider $\|(HA)^n\|$ and prove an appropriate bound that makes the Neumann series convergent.
The point is, of course, that only very large $n$ are relevant here and hence, one may exploit a smallness
property which only shows up after sufficiently many iterations. This is exactly the idea which is used
to solve Volterra equations.
Once one has obtained the existence of $(1-HA)^{-1}$, one rewrites Eq.~\eqref{eq:abstract2} as
\begin{equation}
\label{eq:abstract3}
x=(1-HA)^{-1}Hx_0+(1-HA)^{-1}HF(x)
\end{equation}
and if the nonlinearity $F$ is suitable, it is possible to
solve Eq.~\eqref{eq:abstract3} by a fixed point argument.
This is roughly speaking the program we are going to follow in order to solve Eq.~\eqref{eq:transport}.

\subsection{Time decay of the inhomogeneous term}
According to the program outlined at the beginning of subsection \ref{subsec:solmain}, we first focus
on the difficult linear terms on the right-hand side of Eq.~\eqref{eq:transport}.
In fact, the linear term with the least decay is the one containing the derivative $\hat{\mathcal{D}}\underline{x}$ since
the other one comes with a prefactor of $\tau^{-2}$ which is enough to treat it directly
with the Banach fixed point theorem.
Thus, for the moment we focus on the equation
\begin{equation}
\label{eq:iterate}
[\hat{\mathcal{D}}^2+\beta \hat{\mathcal{D}}+\underline{\xi}]\underline{x}=\underline{e}
-2\beta \mathcal{K} \hat {\mathcal{D}} \underline{x}
\end{equation}
where
\[ \underline{e}(\tau,\xi):=\kappa(\tau)^{-2}\mathcal{F}[|\cdot|\tilde \chi(\tau,\cdot) e_2(\tau,\cdot)](\xi) \]
is the inhomogeneous term on the right-hand side of Eq.~\eqref{eq:transport}.
The first step, however, is to identify suitable spaces in which we intend to solve Eq.~\eqref{eq:iterate}.
It is clear that we have to solve for the pair $(\underline{x},\hat{\mathcal{D}}\underline{x})$ since
both terms $\underline{x}$ and $\hat{\mathcal{D}}\underline{x}$ appear on the right-hand side of Eq.~\eqref{eq:transport}.
The estimates in Lemma~\ref{lem:estH} suggest to place $(\underline{x},\hat{\mathcal{D}}\underline{x})$ in
$\mathcal{X}^{\alpha_1,\beta_1}\times \mathcal{Y}^{\alpha_2,\beta_2}$ where the decay rates $\alpha_j$ and
$\beta_j$, $j=1,2$, are dictated by the inhomogeneous term $\underline{e}$. For the latter we have

\begin{lem}
\label{lem:inhomog}
We have the estimates
\begin{align*}
|\kappa(\tau)^{-2}\mathcal{U}\big (|\cdot|\tilde \chi(\tau,\cdot) e_2(\tau,\cdot)\big )(\xi_d)|
&\leq C_\nu \tau^{-3+\frac12(1+\frac{1}{\nu})+\epsilon} \\
|\kappa(\tau)^{-2}\mathcal{U}\big (|\cdot|\tilde \chi(\tau,\cdot)e_2(\tau,\cdot)\big )(\xi)|
&\leq C_\nu \tau^{-3+\frac12(1+\frac{1}{\nu})+\epsilon} \langle \xi \rangle^{-1}
\end{align*}
for all $\tau\gtrsim 1$, $\xi\geq 0$ and any fixed $\epsilon>0$.
\end{lem}

\begin{proof}
According to Proposition~\ref{prop1}
 we have the bound
\begin{align*}
|\kappa(\tau)^{-2}\tilde \chi(\tau,R)R e_2(\tau,R)|&=|\kappa(\tau)^{-2}\chi(\tfrac{R}{\nu \tau})
R e_2(\tfrac{\nu \tau}{\kappa(\tau)},\tfrac{R}{\kappa(\tau)})| \\
&\leq C_\nu \kappa(\tau)^{\frac12}\tau^{-4}\langle R \rangle^\epsilon
\end{align*}
for some fixed (but arbitrary) $\epsilon>0$ and with symbol behavior of the derivatives of degree at most two.
If $0<\xi\lesssim 1$ or $\xi=\xi_d$ we have from Theorem \ref{thm:spec} the bound $|\phi(R,\xi)|\lesssim 1$
for all $R\geq 0$ and thus,
\begin{align*}
|\kappa(\tau)^{-2}\mathcal{U}\big (|\cdot|\tilde \chi(\tau,\cdot) e_2(\tau,\cdot)\big )(\xi)|
&\leq C_\nu \kappa(\tau)^\frac12 \tau^{-4}\int_0^{3\nu \tau} |\phi(R,\xi)|\langle R\rangle^\epsilon dR \\
&\leq C_\nu \tau^{-3+\frac12(1+\frac{1}{\nu})+\epsilon}.
\end{align*}
If $\xi\gtrsim 1$ we exploit the oscillatory behavior of $\phi(R,\xi)$ given in Theorem \ref{thm:spec}
and perform one integration by parts to gain an additional factor $\xi^{-\frac12}$.
This yields the bound
\[
|\kappa(\tau)^{-2}\mathcal{U}\big (|\cdot|\tilde \chi(\tau,\cdot)e_2(\tau,\cdot)\big )(\xi)|\leq C_\nu
\tau^{-4+\frac12(1+\frac{1}{\nu})+\epsilon}\xi^{-1} \]
which implies the claim.
\end{proof}

\begin{lem}
\label{lem:inhomY}
Let $\alpha,\beta<3-\frac12(1+\frac{1}{\nu})$. Then the function
$\underline{e}$ belongs
to the space $\mathcal{Y}^{\alpha,\beta}$.
\end{lem}

\begin{proof}
The stated time decay (which implies the conditions on $\alpha$ and $\beta$) is an immediate
consequence of Lemma \ref{lem:inhomog}.
Based on the estimate in Lemma \ref{lem:inhomog} it therefore suffices to prove that
$\langle\cdot \rangle^{-1}\in Y$.
It is clear that $\langle \cdot \rangle^{-1}\in L^p(0,\infty)$ for $p$ large and for the
$L^2$-based component we distinguish between small and large $\xi$.
For small $\xi$ we recall that $\rho(\xi)\simeq \xi^{-\frac12}$ (Theorem \ref{thm:spec}) which
is integrable near $0$ and for large $\xi$ we have
\[ |\langle \xi \rangle^{-2}\langle \xi\rangle^{\frac14}\rho(\xi)|\lesssim \langle \xi \rangle^{-\frac54} \]
since $\rho(\xi)\simeq \xi^\frac12$ for $\xi\gtrsim 1$ again by Theorem \ref{thm:spec}.
\end{proof}

Lemma \ref{lem:inhomY} shows that $\underline{e} \in Y^{\frac73+\epsilon,\frac73+\epsilon}$ for a sufficiently small $\epsilon>0$
provided $\nu$ is sufficiently large
which we assume from now on.
Consequently, Lemma \ref{lem:estH} yields
\[ \mathcal{H} \underline{e}\in \mathcal{X}^{\frac73+\epsilon,\frac43+\epsilon-2\delta},\quad
\hat{\mathcal{D}}\mathcal{H} \underline{e} \in \mathcal{Y}^{\frac73+\epsilon,\frac43+\epsilon} \]
and thus, if we choose $\mathcal{X}^{\frac43-2\delta,\frac43-2\delta} \times
\mathcal{Y}^{\frac43,\frac43}$ as our solution
space, we even obtain smallness for the inhomogeneous term, i.e.,
\begin{equation}
\label{eq:He}
\|\mathcal{H}\underline{e}\|_{\mathcal{X}^{\frac43-2\delta,\frac43-2\delta}}\lesssim \nu^{2\delta}\tau_0^{-\epsilon},\quad
\|\hat{\mathcal{D}}\mathcal{H} \underline{e}\|_{\mathcal{Y}^{\frac43,\frac43}}\lesssim \tau_0^{-\epsilon}.
\end{equation}

By applying the operator $\mathcal{H}$ followed by $\hat{\mcD}$ to Eq.~\eqref{eq:iterate}, we obtain
\begin{equation}
\label{eq:iterate2}
\hat{\mcD}\mbx=\hat{\mcD}\mcH \mbe-2\hat{\mcD}\mcH \beta \mcK \hat{\mcD} \mbx.
\end{equation}
Solving this equation for $\hat {\mcD}\mbx \in \mcY^{\frac43,\frac43}$ amounts to proving existence
(and boundedness) of the operator
$(1+2\hat{\mcD}\mcH \beta \mcK)^{-1}$ on $\mcY^{\frac43,\frac43}$.
As in \cite{KST} and \cite{DonKri12} we write
\[  \mcK=\left (\begin{array}{cc}\mcK_{dd} & \mcK_{dc} \\
\mcK_{cd} & \mcK_{cc} \end{array} \right ) \]
for the matrix components of $\mcK$.
With this notation we have
\begin{equation}
\label{eq:matrixDHK}
\hat{\mcD}\mcH \beta \mcK= \left (\begin{array}{cc}\hat{\mcD}_d \mcH_d & 0 \\
0 & \hat{\mcD}_c \mcH_c \end{array} \right )
\left (\begin{array}{cc}\beta \mcK_{dd} & \beta \mcK_{dc} \\
\beta \mcK_{cd} & \beta \mcK_{cc} \end{array} \right )
\end{equation}
where $\hat{\mcD}_d$ is just $\partial_\tau$.
We start by inverting the diagonal elements of $1+2\hat{\mcD}\mcH \beta \mcK$.
Since $\mcK_{dd}$ is a linear map from $\C$ to $\C$, it is just given by a number (to be precise,
we have $\mcK_{dd}a=-\frac32 a$ for all $a\in \C$, see \cite{DonKri12}).
Furthermore, by Lemma \ref{lem:estH} we have
\[ |\hat {\mcD}_d \mcH_d \beta \mcK_{dd}x_d(\tau)|\lesssim \tau^{-1}|x_d(\tau)|\leq \tau_0^{-1}|x_d(\tau)| \]
since $\beta(\tau)\simeq \tau^{-1}$.
This shows that $(1+2\hat{\mcD}_d \mcH_d \beta \mcK_{dd})^{-1}$ exists.

\subsection{Structure and properties of $\mcK$}
In order to proceed, we need more detailed information on the operator $\mcK$.
The operator $\mcK$ has been analysed in detail in \cite{KST} and \cite{DonKri12}.
It is easy to see that $\mcK_{cd}: \C \to L^2_\rho(0,\infty)$ is given by
\[ \mcK_{cd}a(\xi)=\frac{a}{\|\phi(\cdot,\xi_d)\|_{L^2(0,\infty)}^2}\int_0^\infty \phi(R,\xi)[R\partial_R-1]\phi(R,\xi_d)dR \]
with $\phi$ from Theorem \ref{thm:spec}.
For $\mcK_{dc}$ and $\mcK_{cc}$
we recall the following result.

\begin{thm}
\label{thm:K}
The operator $\mcK_{cc}: L^2_\rho(0,\infty)\to L^2_\rho(0,\infty)$ is given by
\[ \mcK_{cc}f(\xi)=\int_0^\infty K_{cc}(\xi,\eta)f(\eta)d\eta \]
where the kernel $K_{cc}$ is of the form
\[ K_{cc}(\xi,\eta)=\frac{\rho(\eta)}{\xi-\eta}F(\xi,\eta) \]
with a symmetric function $F \in C^2((0,\infty)\times (0,\infty))$.
Furthermore, for any $N\in \mathbb{N}$, $F$ satisfies the bounds
  \[\begin{split}
   | F(\xi,\eta)| &\leq C_N \left\{ \begin{array}{cc} \xi+\eta &
       \xi+\eta \leq 1 \cr (\xi+\eta)^{-1} (1+|\xi^\frac12
       -\eta^\frac12|)^{-N} & \xi+\eta \geq 1
     \end{array} \right.\\
   | \partial_{\xi} F(\xi,\eta)|+| \partial_{\eta} F(\xi,\eta)| &\leq C_N \left\{
     \begin{array}{cc} 1 & \xi+\eta \leq 1 \cr (\xi+\eta)^{-\frac32}
       (1+|\xi^\frac12 -\eta^\frac12|)^{-N} & \xi+\eta \geq 1
     \end{array} \right.\\
  \max_{j+k=2} | \partial^j_{\xi}\partial^k_{\eta} F(\xi,\eta)| &\leq C_N \left\{
     \begin{array}{cc} (\xi+\eta)^{-\frac12} & \xi+\eta \leq 1 \cr
       (\xi+\eta)^{-2} (1+|\xi^\frac12 -\eta^\frac12|)^{-N} &
       \xi+\eta \geq 1
     \end{array} \right. .
     \end{split}
   \]
Finally, the operator $\mcK_{dc}: L^2_\rho(0,\infty)\to \C$ is of the form
\[ \mcK_{dc}f=\int_0^\infty K_{dc}(\xi)f(\xi)\rho(\xi)d\xi \]
with a smooth and rapidly decreasing function $K_{dc}$.
\end{thm}

\begin{proof}
See \cite{KST}, Theorem 5.1.
\end{proof}

As a consequence of Theorem \ref{thm:K} we have the following mapping properties of $\mcK_{cc}$
and the commutator $[\mcA_c,\mcK_{cc}]$ where we recall that
\[ \mcA_c f(\xi)=-2\xi f'(\xi)-(\tfrac52+\tfrac{\xi\rho'(\xi)}{\rho(\xi)})f(\xi). \]
\begin{prop}
\label{prop:KAKc}
We have the bounds
\begin{align*}
\|\mcK_{cc}f\|_X &\lesssim \|f\|_X & \|[\mcA_c,\mcK_{cc}]f\|_X&\lesssim \|f\|_X \\
\|\mcK_{cc}f\|_Y &\lesssim \|f\|_X & \|[\mcA_c,\mcK_{cc}]f\|_Y&\lesssim \|f\|_X \\
\|\mcK_{cc}g\|_Y &\lesssim \|g\|_Y & \|[\mcA_c,\mcK_{cc}]g\|_Y&\lesssim \|g\|_Y
\end{align*}
for all $f\in X$ and $g\in Y$.
\end{prop}

\begin{proof}
This follows from the representation in Theorem \ref{thm:K} but requires some harmonic analysis.
We refer the reader to \cite{DonKri12}, Propositions 5.5 and 5.8 for the proof.
We remark that the bounds for $[\mcA_c,\mcK_{cc}]$ can be obtained in the same fashion as
the ones for $\mcK_{cc}$ by
noting that the kernel of $[\mcA_c,\mcK_{cc}]$ is of the form
\[ \frac{\rho(\eta)}{\xi-\eta}\tilde F(\xi,\eta) \]
with
\[\tilde F(\xi,\eta)=\frac{\eta\rho'(\eta)}{\rho(\eta)}F(\xi,\eta)+[\xi
\partial\xi+\eta \partial_\eta]F(\xi,\eta), \]
see \cite{KST}, p.~52.
\end{proof}

In what follows it is necessary to split the operator $\mcK_{cc}$ into a diagonal and an off-diagonal
part. Thus, for $n_0\in \N$ we set
\[ K_{n_0}^d(\xi,\eta)=\chi \left (n_0(\tfrac{\xi}{\eta}-1)\right )K_{cc}(\xi,\eta) \]
where $\chi$ is a standard smooth cut-off with $\chi(x)=1$ for $|x|\leq 1$ and $\chi(x)=0$ for $|x|\geq 2$.
Furthermore, we denote by $\mcK_{n_0}^d$ the corresponding operator and write
\[ \mcK_{n_0}^{nd}:=\mcK_{cc}-\mcK_{n_0}^d \]
for the off-diagonal part.
First, we establish a smoothing estimate for the off-diagonal part at small frequencies.

\begin{lem}
\label{lem:Kccnd}
With $p$ from Definition \ref{def:spaces}, we have the bounds
\begin{align*}
\| \, |\cdot|^{-\frac{1}{2p}}\langle \cdot \rangle^{\frac{1}{2p}} \mcK_{n_0}^{nd}f\|_Y&\lesssim n_0^2\|f\|_Y \\
\| \, |\cdot|^{-\frac{1}{2p}}\langle \cdot \rangle^{\frac{1}{2p}} \mcK_{n_0}^{nd}f\|_X&\lesssim n_0^4\|f\|_X
\end{align*}
for all $n_0\in \N$, $n_0\geq 100$.
\end{lem}

\begin{proof}
Explicitly, the operator $\mcK_{n_0}^{nd}$ is given by
\[ \mcK_{n_0}^{nd}f(\xi)=\int_0^\infty K_{n_0}^{nd}(\xi,\eta)f(\eta)d\eta \]
with
\[ K_{n_0}^{nd}(\xi,\eta)=\left [1-\chi \left (n_0(\tfrac{\xi}{\eta}-1)\right )
\right ] K_{cc}(\xi,\eta). \]
Note that on the support of $K_{n_0}^{nd}$ we have $|\frac{\xi}{\eta}-1|\geq \frac{1}{n_0}$
and thus, either $\eta\leq \frac{1}{1+\frac{1}{n_0}}\xi$ or $\eta\geq \frac{1}{1-\frac{1}{n_0}}\xi$.
In the former case we obtain
\[ \xi-\eta\geq \left (1-\tfrac{1}{1+\frac{1}{n_0}}\right )\xi\simeq \tfrac{1}{n_0}\xi
\gtrsim \tfrac{1}{n_0}(\xi+\eta) \]
and in the latter, $\eta-\xi\gtrsim \frac{1}{n_0}(\xi+\eta)$.
Thus, we have $|\xi-\eta|\gtrsim \frac{1}{n_0}|\xi+\eta|$ and from Theorem \ref{thm:K} we obtain
the bound
$|K_{n_0}^{nd}(\xi,\eta)|\lesssim n_0 \eta^{-\frac12}$ provided $\xi+\eta\leq 1$.
If $\xi+\eta\geq 1$ we note that, as before, $|\xi^\frac12-\eta^\frac12|\gtrsim n_0^{-\frac12}\xi^\frac12$
and also $|\xi^\frac12-\eta^\frac12|\gtrsim n_0^{-\frac12}\eta^\frac12$.
Thus, from Theorem \ref{thm:K} we obtain the bound
$|K_{n_0}^{nd}(\xi,\eta)|\lesssim n_0^2 \langle \xi \rangle^{-1}\langle \eta \rangle^{-2}$.
We conclude that
\[ |\tilde K(\xi,\eta)|:=|\xi^{-\frac{1}{2p}}\langle \xi \rangle^{\frac{1}{2p}}K_{n_0}^{nd}(\xi,\eta)|\lesssim
n_0^2\xi^{-\frac{1}{2p}}\langle \xi \rangle^{\frac{1}{2p}}
\langle \xi\rangle^{-1}\eta^{-\frac12}\langle \eta \rangle^\frac12
\langle \eta \rangle^{-\frac32} \]
for all $\xi,\eta\geq 0$ and thus,
\[ \|\tilde K\|_{L^{p}(0,\infty)L^{p'}(0,\infty)}\lesssim n_0^2 \]
which implies
\[ \| \, |\cdot|^{-\frac{1}{2p}}\langle \cdot \rangle^{\frac{1}{2p}} \mcK_{n_0}^{nd}f\|_{L^p(0,\infty)}
\lesssim n_0^2\|f\|_{L^p(0,\infty)}. \]
For the weighted $L^2$-component we estimate
\begin{align*}
|\tilde K(\xi,\eta)|:&=|\xi^{-\frac{1}{2p}}\langle \xi \rangle^{\frac{1}{2p}}\langle \xi\rangle^\frac18
\rho(\xi)^\frac12 K_{n_0}^{nd}(\xi,\eta)\langle \eta \rangle^{-\frac18}\rho(\eta)^{-\frac12}| \\
&\lesssim n_0^2 \xi^{-\frac{1}{2p}-\frac14}\langle \xi \rangle^{\frac{1}{2p}+\frac14}
\langle \xi \rangle^{-\frac58}\eta^{-\frac14}\langle \eta \rangle^\frac14 \langle \eta \rangle^{-\frac{15}{8}}
\end{align*}
which implies $\|\tilde K\|_{L^2(0,\infty)L^2(0,\infty)}\lesssim n_0^2$ and the claim follows.

For the second bound we proceed completely analogous with the exception that the
$L^2$-component gets estimated in a slightly different way and we use the stronger bound
$|K(\xi,\eta)|\lesssim n_0^4 \langle \xi\rangle^{-2}\langle \eta\rangle^{-2}$ from Theorem
\ref{thm:K}.
With
\[ \tilde K(\xi,\eta):=\xi^{-\frac{1}{2p}}\langle \xi\rangle^{\frac{1}{2p}}
\xi^{\frac12-\delta}\langle \xi\rangle^{-\frac12+\delta}K_{n_0}^{nd}(\xi,\eta)\eta^{-\frac12+\delta}
\langle \eta \rangle^{\frac12-\delta}
\]
we obtain $\|\tilde K\|_{L^p(0,\infty)L^{p'}(0,\infty)}\lesssim n_0^4$ provided
$(-1+\delta)p'>-1$ which we may safely assume since $p$ is supposed to be large.
For the $L^2$-component in the second bound we consider the kernel
\[ \tilde K(\xi,\eta):=\xi^{-\frac{1}{2p}}\langle \xi \rangle^{\frac{1}{2p}}
\xi^\frac12 \langle \xi\rangle^\frac18 \rho(\xi)^\frac12 K_{n_0}^{nd}(\xi,\eta)
\eta^{-\frac12+\delta}\langle \eta \rangle^{\frac12-\delta} \]
which satisfies the bound
\[ |\tilde K(\xi,\eta)|\lesssim n_0^4 \xi^{-\frac{1}{2p}-\frac14}\langle \xi\rangle^{\frac{1}{2p}+\frac14}
\langle \xi\rangle^{-\frac98}\eta^{-1+\delta}\langle \eta \rangle^{1-\delta}\langle \eta\rangle^{-2}. \]
This implies the bound $\|\tilde K\|_{L^2(0,\infty)L^{p'}(0,\infty)}\lesssim n_0^4$ which
concludes the proof.
\end{proof}

We also need a corresponding smoothing property for the diagonal part.
Here it is crucial for the following that the obtained bound does not depend on $n_0$.
We start with an estimate for a truncated version of the Hilbert transform.

\begin{lem}
\label{lem:HT}
Let $H_n$, $n\in \N$, be given by
\[ H_n f(\xi):=\int_0^\infty \frac{\chi(n(\tfrac{\xi}{\eta}-1))}{\xi-\eta}f(\eta)d\eta,
\quad \xi\geq 0 \]
where $\chi$ is a smooth cut-off function satisfying $\chi(x)=1$ for $|x|\leq 1$
and $\chi(x)=0$ for $|x|\geq 2$.
Then $H_n$ extends
to a bounded operator on $L^q(0,\infty)$ for any $q\in (1,\infty)$ and we have
\[ \|H_n f\|_{L^q(0,\infty)}\lesssim \|f\|_{L^q(0,\infty)} \]
for all $f \in L^q(0,\infty)$ and all $n \geq 100$.
\end{lem}

\begin{proof}
We use
\[ \chi(n(\tfrac{\xi}{\eta}-1))=1+\frac{n(\xi-\eta)}{\eta}\int_0^1 \chi'(n s(\tfrac{\xi}{\eta}-1))ds \]
to decompose the kernel according to
\begin{equation}
\label{eq:decompHT}
 \frac{\chi(n(\tfrac{\xi}{\eta}-1))}{\xi-\eta}=\frac{1}{\xi-\eta}+\frac{n}{\eta}O(1).
 \end{equation}
Let
$I_{jk}^n:=\left [2^{j-1}+3\tfrac{k-1}{n}2^{j-1},2^{j-1}+3\tfrac{k}{n}2^{j-1}\right ]$
and observe that
\[ [2^{j-1},2^{j+1}]=\bigcup_{k=1}^n I_{jk}^n. \]
Furthermore,
set $\Delta_n:=\{(\xi,\eta) \in [0,\infty)^2: \chi(n(\tfrac{\xi}{\eta}-1))\not=0\}$.
Since we have $|\xi-\eta|\leq \frac{2}{n}\eta$ for all $(\xi,\eta)\in \Delta_n$,
$\eta \in I_{jk}^n$ implies $\xi\in \tilde I_{jk}^n$ for all $(\xi,\eta)\in \Delta_n$
where $\tilde I_{jk}^n$ are suitable (overlapping) intervals with
$|\tilde I_{jk}^n|\simeq \frac{2^j}{n}$ and $[2^{j-2},2^{j+2}]=\bigcup_{k=1}^n \tilde I_{jk}^n$.
As a consequence, we infer
\[ \Delta_n \subset \bigcup_{j\in\Z}\bigcup_{k=1}^n \tilde I_{jk}^n\times I_{jk}^n \]
for any $n\geq 100$.
Thus, we obtain
\begin{align*}
H_n f(\xi)&=\tfrac12 \sum_{j\in \Z}\sum_{k=1}^n \int_0^\infty
\frac{\chi(n(\tfrac{\xi}{\eta}-1))}{\xi-\eta}1_{I_{jk}^n}(\eta)f(\eta)d\eta \\
&=\tfrac12 \sum_{j\in \Z}\sum_{k=1}^n 1_{\tilde I_{jk}^n}(\xi)\int_0^\infty
\frac{\chi(n(\tfrac{\xi}{\eta}-1))}{\xi-\eta}1_{I_{jk}^n}(\eta)f(\eta)d\eta \\
&=\tfrac12 \sum_{j\in\Z}\sum_{k=1}^n 1_{I_{jk}^n}(\xi)H_n(1_{I_{jk}^n}f)(\xi).
\end{align*}
Consequently, it suffices to bound the operator $f\mapsto 1_{\tilde I_{jk}^n}
H_n(1_{I_{jk}^n}f)$ on $L^q:=L^q(0,\infty)$,
uniformly in $n\geq 100$ and $j\in \Z$, because then we can conclude that
\begin{align*}
\|H_n f\|_{L^q}^q &\lesssim \sum_{j\in\Z}\sum_{k=1}^n \|1_{\tilde I_{jk}^n}
H_n(1_{I_{jk}^n}f)\|_{L^q}^q
=\sum_{j\in\Z}\sum_{k=1}^n \|1_{\tilde I_{jk}^n}H_n(1_{I_{jk}^n}1_{I_{jk}^n}f)\|_{L^q}^q \\
&\lesssim \sum_{j\in\Z}\sum_{k=1}^n \|1_{I_{jk}^n}f\|_{L^q}^q\lesssim \|f\|_{L^q}^q.
\end{align*}
According to Eq.~\eqref{eq:decompHT},
the kernel of the operator $f \mapsto 1_{\tilde I_{jk}^n}H_n(1_{I_{jk}^n}f)$ is of the form
\begin{equation*}
 1_{\tilde I_{jk}^n}(\xi)1_{I_{jk}^n}(\eta)\frac{\chi(n(\tfrac{\xi}{\eta}-1))}{\xi-\eta}=
\frac{1_{\tilde I_{jk}^n}(\xi)1_{I_{jk}^n}(\eta)}{\xi-\eta}+n2^{-j}1_{\tilde I_{jk}^n}(\xi)
1_{I_{jk}^n}(\eta)O(1).
\end{equation*}
Thus, we obtain
the decomposition
\[ 1_{\tilde I_{jk}^n}H_n(1_{I_{jk}^n}f)=\pi 1_{\tilde I_{jk}^n}H(1_{I_{jk}^n} f)+B_{jk}^n f \]
with the standard Hilbert transform
$H$ and the kernel of $B_{jk}^n$ is pointwise bounded by
$Cn2^{-j}1_{\tilde I_{jk}^n}(\xi)1_{I_{jk}^n}(\eta)$ for some absolute constant $C>0$.
We immediately obtain $\|1_{\tilde I_{jk}^n}H(1_{I_{jk}^n} f)\|_{L^q}\lesssim \|f\|_{L^q}$ by the $L^q$-boundedness of
the Hilbert transform for $q\in (1,\infty)$ and the operator norm of $B_{jk}^n$ is bounded by
\[ \|B_{jk}^n\|_{L^q}\lesssim n2^{-j}
\left (\int_0^\infty 1_{\tilde I_{jk}^n}(\xi)d\xi \right )^{1/q}
\left (\int_0^\infty 1_{I_{jk}^n}(\eta)d\eta \right )^{1/q'} \lesssim 1 \]
for all $n\geq 100$, $j\in \Z$ and $k\in \{1,2,\dots,n\}$ since $|\tilde I_{jk}^n|
\simeq |I_{jk}^n|\simeq \frac{2^j}{n}$.
\end{proof}

With this result at our disposal, we can now prove the desired smoothing property
of $\mcK_{n_0}^d$.

\begin{lem}
\label{lem:Kccd}
For any $\epsilon>0$, $a,b\in \R$, and $q\in (1,\infty)$ we have the bound
\begin{align*}
\| \, |\cdot|^{-\frac{1}{2}+\epsilon}\langle \cdot \rangle^{1-2\epsilon}\mcK_{n_0}^d f|\cdot|^a
\langle \cdot\rangle^b\|_{L^q(0,\infty)}
&\lesssim \|f|\cdot|^a \langle \cdot\rangle^b\|_{L^q(0,\infty)}
\end{align*}
for all $n_0 \geq 100$.
\end{lem}

\begin{proof}
Consider the operator $\mcJ$ with kernel
\[ \xi^{-\frac{1}{2}+\epsilon}\langle \xi \rangle^{1-2\epsilon}\xi^a \langle \xi \rangle^b
K_{n_0}^d(\xi,\eta)\eta^{-a}\langle \eta \rangle^{-b}. \]
In order to prove the assertion it suffices to show that $\mcJ$ extends to an operator
on $L^q:=L^q(0,\infty)$ for $q \in (1,\infty)$
which is uniformly bounded in $n_0\geq 100$.
According to Theorem \ref{thm:K}, the kernel of $\mcJ$ can be written in the form
\[\xi^{-\frac{1}{2}+\epsilon}\langle \xi \rangle^{1-2\epsilon}\xi^a \langle \xi \rangle^b
K_{n_0}^d(\xi,\eta)\eta^{-a}\langle \eta \rangle^{-b}=
\chi(n_0(\tfrac{\xi}{\eta}-1))\frac{G(\xi,\eta)}{\xi-\eta} \]
where
\[ G(\xi,\eta)= \xi^{-\frac{1}{2}+\epsilon}\langle \xi \rangle^{1-2\epsilon}
\xi^a\langle \xi\rangle^b\rho(\eta)F(\xi,\eta)\eta^{-a}\langle \eta\rangle^{-b}. \]
We decompose $\mcJ=\mcJ_1+\mcJ_2$ where
\begin{align*}
 \mcJ_1 f(\xi)&=\int_0^\infty \chi(n_0(\tfrac{\xi}{\eta}-1))\frac{G(\eta,\eta)}{\xi-\eta}f(\eta)d\eta \\
 \mcJ_2 f(\xi)&=\int_0^\infty \chi(n_0(\tfrac{\xi}{\eta}-1))
 \frac{G(\xi,\eta)-G(\eta,\eta)}{\xi-\eta}f(\eta)d\eta.
 \end{align*}
 By setting $g(\eta):=G(\eta,\eta)$ we see that $\mcJ_1 f=H_{n_0}(gf)$ where $H_{n_0}$
 is the truncated Hilbert transform from Lemma \ref{lem:HT}.
Note that Theorem \ref{thm:K} implies $\|g\|_{L^\infty(0,\infty)}\lesssim 1$ and thus,
 \[ \|\mcJ_1 f\|_{L^q}= \|H_{n_0}(gf)\|_{L^q}\lesssim \|gf\|_{L^q}\lesssim \|f\|_{L^q} \]
 for all $n_0\geq 100$ by Lemma \ref{lem:HT}.
Consequently, it suffices to consider the operator $\mcJ_2$.

First, we study the case $\xi,\eta \leq 4$.
Since
\[ |G(\xi,\eta)-G(\eta,\eta)|\leq |\xi-\eta|\int_0^1 |\partial_1 G(\eta+s(\xi-\eta),\eta)|ds \]
we obtain from Theorem \ref{thm:K} the estimate
\begin{align*}
A_1(\xi,\eta)&:=1_{[0,4]}(\xi)1_{[0,4]}(\eta)\chi(n_0(\tfrac{\xi}{\eta}-1))\left|
\frac{G(\xi,\eta)-G(\eta,\eta)}{\xi-\eta}\right | \\
&\lesssim  1_{[0,4]}(\xi)1_{[0,4]}(\eta)\chi(n_0(\tfrac{\xi}{\eta}-1))\eta^{-1+\epsilon} \\
&\lesssim 1_{[0,4]}(\xi)1_{[0,4]}(\eta)\xi^{\frac{1}{q}(-1+\epsilon)}\eta^{\frac{1}{q'}(-1+\epsilon)}
\end{align*}
which yields $\|A_1\|_{L^qL^{q'}}\lesssim 1$ for all $n_0\geq 100$ and any $q\in (1,\infty)$.

It remains to study the case $\xi,\eta \in \Omega:=[0,\infty)^2\backslash [0,4]^2$.
Here we further distinguish between $|\xi-\eta|\leq 1$ and $|\xi-\eta|\geq 1$.
In the former case we obtain from Theorem \ref{thm:K} the bound
\begin{align*} A_2(\xi,\eta)&:=1_{[-1,1]}(\xi-\eta)1_\Omega(\xi,\eta)\chi(n_0(\tfrac{\xi}{\eta}-1))
\left|
\frac{G(\xi,\eta)-G(\eta,\eta)}{\xi-\eta}\right | \\
&\lesssim 1_{[-1,1]}(\xi-\eta)1_\Omega(\xi,\eta)\eta^{-\epsilon}.
\end{align*}
We define $J_k:=[k+1,k+3]$, $\tilde J_k:=[k,k+4]$ and note that
\[ \Delta \subset \bigcup_{k=1}^\infty \tilde J_k \times J_k \]
where $\Delta:=\{(\xi,\eta)\in \Omega: |\xi-\eta|\leq 1\}$.
Since
\[ 1_{J_k}(\eta)1_{[-1,1]}(\xi-\eta)=1_{\tilde J_k}(\xi)1_{J_k}(\eta)1_{[-1,1]}(\xi-\eta) \]
it suffices to consider the kernel $A_2$ on $\tilde J_k \times J_k$ (cf.~the proof of Lemma \ref{lem:HT}).
We obtain $\|A_2\|_{L^q(\tilde J_k)L^{q'}(J_k)}\lesssim 1$
for all $k\in \N$ and all $n_0\geq 100$ which settles the case $|\xi-\eta|\leq 1$.
Finally, if $|\xi-\eta|\geq 1$, we define dyadic intervals $I_N:=[2^{N-1},2^{N+1}]$,
$\tilde I_N:=[2^{N-2},2^{N+2}]$ and consider the
kernel on $\tilde I_N \times I_N$, $N\in \N$.
Thanks to the cut-off $\chi(n_0(\tfrac{\xi}{\eta}-1))$ it suffices to bound
\[ A_3(\xi,\eta):=1_{[1,\infty)}(|\xi-\eta|)1_{\tilde I_N}(\xi)1_{I_N}(\eta)\chi(n_0(\tfrac{\xi}{\eta}-1))
\frac{G(\xi,\eta)-G(\eta,\eta)}{\xi-\eta}, \]
uniformly in $N\in\N$.
Note that $1\leq |\xi-\eta|\leq 2\eta\leq 2^{N+2}$ on the support of $A_3$.
We further subdivide this interval by $[1,2^{N+2}]=\bigcup_{j=1}^{N+1}I_j$
 and from Theorem \ref{thm:K} we obtain the bound
\begin{equation}
\label{eq:A3}
 |A_3(\xi,\eta)|\lesssim 2^{-\epsilon N}\sum_{j=1}^{N+1}A_{3j}(\xi,\eta)
 \end{equation}
where
\[ A_{3j}(\xi,\eta)=1_{I_j}(|\xi-\eta|)
1_{\tilde I_N}(\xi)1_{I_N}(\eta)\chi(n_0(\tfrac{\xi}{\eta}-1))2^{-j}. \]
Thanks to the cut-off $1_{I_j}(|\xi-\eta|)$ it suffices to bound $A_{3j}$ on squares
$Q_j$ of area $\simeq 2^{2j}$ which yields $\|A_{3j}\|_{L^qL^{q'}(Q_j)}\lesssim 1$ for all
$j \in \{1,2,\dots,N+1\}$ and any $q\in(1,\infty)$.
Consequently, by Eq.~\eqref{eq:A3} we obtain
\[ \|A_3\|_{L^q(\tilde I_N)L^{q'}(I_N)}\lesssim N 2^{-\epsilon N} \lesssim 1 \]
for all $N\in \N$ which finishes the proof.
\end{proof}

\subsection{Estimates for the off-diagonal part}
Recall that our aim is to prove smallness of $(\hat{\mcD}_c \mcH_c \beta \mcK_{cc})^n$
for sufficiently large $n$.
As suggested by the decomposition $\mcK_{cc}=\mcK_{n_0}^d+\mcK_{n_0}^{nd}$ we consider the
diagonal and off-diagonal parts separately.
In fact, it turns out that for the off-diagonal part it suffices to consider the
operator $\hat{\mcD}_c\mcH_c \beta \mcK_{n_0}^{nd}
\hat{\mcD}_c\mcH_c \beta$, i.e., $\mcK_{n_0}^{nd}$ gets ``sandwiched'' between
two copies of $\hat{\mcD}_c \mcH_c \beta$.
Our goal is to show that the norm (on $\mcY^\alpha$) of this operator can be made
small by choosing $\tau_0$ in Definition \ref{def:spaces} large.
More precisely, we have the following result.

\begin{lem}
\label{lem:offdiag}
Let $\alpha>\frac34(1+\frac{1}{\nu})$. Then there exists an $\epsilon>0$ such that
\[ \|\hat{\mcD}_c \mcH_c \beta \mcK_{n_0}^{nd}\hat{\mcD}_c \mcH_c \beta
\|_{\mcY^\alpha}\lesssim n_0^4 \tau_0^{-\epsilon} \]
for all $n_0\in \N$ where $\tau_0$ is from Definition \ref{def:spaces}.
\end{lem}

\begin{proof}
We have
\begin{align}
\label{eq:HKH}
\hat{\mcD}_c &\mcH_c \beta \mcK_{n_0}^{nd}\hat{\mcD}_c \mcH_c \beta x(\tau,\xi) \nonumber \\
&=\int_\tau^\infty \hat{H}_c(\tau,\sigma_1,\xi)\beta(\sigma_1)
\int_0^\infty K_{n_0}^{nd}(\omega(\tau, \sigma_1)^2 \xi,\eta) \nonumber \\
&\quad \times \int_{\sigma_1}^\infty \hat{H}_c(\sigma_1,\sigma_2,\eta)\beta(\sigma_2)
x(\sigma_2, \omega(\sigma_1, \sigma_2)^2\eta)d\sigma_2 d\eta d\sigma_1
\end{align}
where $\omega(s_1, s_2)=\kappa(s_1)\kappa^{-1}(s_2)$ and $K_{n_0}^{nd}$ is the kernel of the operator
$\mcK_{n_0}^{nd}$.
We split the integral over $\sigma_1$ in two parts by distinguishing between the cases
$\sigma_1 \xi\lesssim 1$
and $\sigma_1 \xi\gtrsim 1$.
In the former case we exploit the smoothing property from Lemma \ref{lem:Kccnd} in order
to gain a small factor.
Thus, we write $y(\sigma,\xi):=\hat{\mcD}_c\mcH_c \beta x(\sigma,\xi)$ and note
that $y \in \mcY^{\alpha}$
by Lemma \ref{lem:estH}.
We have to estimate
\begin{align*}
\mcJ_1 y(\tau,\xi):=\int_\tau^\infty \chi(\sigma \xi) \hat H_c(\tau,\sigma,\xi)\beta(\sigma)
\mcK_{n_0}^{nd}y(\sigma,\omega(\tau, \sigma)^2\xi)d\sigma.
\end{align*}
Recall that
\[ \hat H_c(\tau,\sigma,\xi)=-\omega(\tau, \sigma)^\frac32\rho(\xi)^{-\frac12}
\rho \left (\omega(\tau,\sigma)^2 \xi \right )^\frac12
\cos \left (\kappa(\tau)\xi^{\frac{1}{2}}\int_{\tau}^\sigma \kappa^{-1}(u)\,du \right ) \]
and, since $\omega(\tau, \sigma)\leq 1$, we obtain from the asymptotics of
$\rho$ in Theorem \ref{thm:spec} the bound
$|\hat H_c(\tau,\sigma,\xi)|\lesssim \omega(\tau, \sigma)$.
Consequently, with $p$ from Definition \ref{def:spaces} we have
\begin{align*} |\mcJ_1 y(\tau,\xi)|&\lesssim \int_\tau^\infty \chi(\sigma \xi)\sigma^{-1}
\omega(\tau, \sigma)[\omega(\tau, \sigma)^2 \xi]^{\frac{1}{2p}}
[\omega(\tau, \sigma)^2 \xi]^{-\frac{1}{2p}}|\mcK_{n_0}^{nd}y(\sigma,\omega(\tau, \sigma)^2 \xi)|d\sigma  \\
&\lesssim \int_\tau^\infty \sigma^{-1-\frac{1}{2p}}
\omega(\tau, \sigma)^{1+\frac{1}{p}}
[\omega(\tau, \sigma)^2 \xi]^{-\frac{1}{2p}}|\mcK_{n_0}^{nd}
y(\sigma,\omega(\tau,\sigma)^2 \xi)|d\sigma
\end{align*}
and Lemmas \ref{lem:estB}, \ref{lem:Kccnd} yield
\[ \|\mcJ_1 y\|_{\mcY^\alpha}\lesssim \tau_0^{-\frac{1}{2p}}\sup_{\tau>\tau_0}
\tau^{\alpha}\| \, |\cdot|^{-\frac{1}{2p}}
\mcK_{n_0}^{nd}y(\tau,\cdot)\|_Y\lesssim n_0^2 \tau_0^{-\frac{1}{2p}}\|y\|_{\mcY^{\alpha}}. \]
Thus, by Lemma \ref{lem:estH} we obtain $\|\mcJ_1 y\|_{\mcY^\alpha}
\lesssim n_0^4\tau_0^{-\frac{1}{2p}}\|x\|_{\mcY^\alpha}$.

It remains to consider the case $\sigma_1 \xi\gtrsim 1$.
Unfortunately, this is more complicated and we have to exploit the oscillation of the
kernel.
After the change of variable $\eta\mapsto \omega(\sigma_1, \sigma_2)^{-2}\eta$
and an application of Fubini it remains to study the operator
\begin{align*}
\mcJ_2 x(\tau,\xi):=&\int_0^\infty \int_\tau^\infty \int_{\sigma_1}^\infty
[1-\chi(\sigma_1 \xi)]\hat H_c(\tau,\sigma_1,\xi)
\hat H_c \left (\sigma_1,\sigma_2,\omega(\sigma_1, \sigma_2)^{-2}\eta \right) \\
&\times \beta(\sigma_1)\beta(\sigma_2)\omega(\sigma_1,\sigma_2)^{-2}
K_{n_0}^{nd}\left (\omega(\tau, \sigma_1)^2 \xi,
\omega(\sigma_1, \sigma_2)^{-2}
\eta \right ) \\
&\times x(\sigma_2,\eta)d\sigma_2 d\sigma_1 d\eta,
\end{align*}
cf.~Eq.~\eqref{eq:HKH}.
We have
\begin{align*}
\hat H_c(&\tau,\sigma_1,\xi)
\hat H_c \left (\sigma_1,\sigma_2,\omega(\sigma_1, \sigma_2)^{-2}\eta \right)=
A(\tau,\sigma_1,\sigma_2,\xi,\eta) \\
&\times\cos\left (\kappa(\tau)\xi^{\frac{1}{2}}\int_{\tau}^{\sigma_1}\kappa^{-1}(u)\,du\right )
\cos\left (\kappa(\sigma_2)\eta^{\frac{1}{2}}\int_{\sigma_1}^{\sigma_2}\kappa^{-1}(u)\,du\right )
\end{align*}
where
\begin{align*}
A(\tau,\sigma_1,\sigma_2,\xi,\eta)=&
\omega(\tau, \sigma_2)^{\frac32}\rho(\xi)^{-\frac12}
\rho\left (\omega(\tau, \sigma_1)^2\xi\right )^\frac12
\rho\left (\omega(\sigma_1, \sigma_2)^{-2}\eta\right )^{-\frac12}
\rho(\eta)^\frac12.
\end{align*}
By the asymptotics of $\rho$ given in Theorem \ref{thm:spec} and the fact that
$\tau\leq  \sigma_1\leq \sigma_2$, we obtain the estimate
$|A(\tau,\sigma_1,\sigma_2,\xi,\eta)|\lesssim \omega(\tau, \sigma_2)$
with symbol behavior under differentiation with respect to each variable.
Furthermore, by using the trigonometric identity $2\cos a \cos b=\cos (a+b)+\cos(a-b)$
we observe that the operator in question
decomposes as $\mcJ_2=\mcA_++\mcA_-$
where
\begin{align*} \mcA_{\pm}x(\tau,\xi)=&
\tfrac12 \int_0^\infty \int_\tau^\infty \int_{\sigma_1}^\infty
[1-\chi(\sigma_1 \xi)]\beta(\sigma_1)\beta(\sigma_2)\omega(\sigma_1, \sigma_2)^{-2}
A(\tau,\sigma_1,\sigma_2,\xi,\eta) \\
&\times \cos \left (\xi^{\frac{1}{2}}\kappa(\tau)\int_{\tau}^{\sigma_1}\kappa^{-1}(u)\,du \pm \eta^{\frac{1}{2}}\kappa(\sigma_2)\int_{\sigma_1}^{\sigma_2}\kappa^{-1}(u)\,du\right ) \\
&\times K_{n_0}^{nd}\left (\omega(\tau, \sigma_1)^2 \xi,
\omega(\sigma_1, \sigma_2)^{-2}
\eta \right )x(\sigma_2,\eta)d\sigma_2 d\sigma_1 d\eta.
\end{align*}
It suffices to consider $\mcA_+$.
We abbreviate
\[ \mu:= \xi^{\frac{1}{2}}\kappa(\tau)\kappa^{-1}(\sigma_1)- \eta^{\frac{1}{2}}\kappa(\sigma_2)\kappa^{-1}(\sigma_1) \]
and since $\partial_{\sigma_1}\mu=-c(\sigma_1)(1+\frac{1}{\nu})\sigma_1^{-1}\mu$, $c(\sigma_1)\in [2^{-1},2]$,  we obtain the identity
\begin{align*}
&\cos \left (\Omega\right )=\partial_{\sigma_1}\left [\mu^{-1}
\sin(\Omega) \right ] - c(\sigma_1)(1+\tfrac{1}{\nu})\sigma_1^{-1}\mu^{-1}\sin(\Omega),\\&\Omega = \xi^{\frac{1}{2}}\kappa(\tau)\int_{\tau}^{\sigma_1}\kappa^{-1}(u)\,du + \eta^{\frac{1}{2}}\kappa(\sigma_2)\int_{\sigma_1}^{\sigma_2}\kappa^{-1}(u)\,du
\end{align*}
Then we use the integration by parts formula
\begin{align*}
\int_\tau^\infty &\int_{\sigma_1}^\infty \partial_{\sigma_1}f(\sigma_1,\sigma_2)
g(\sigma_1,\sigma_2)d\sigma_2 d\sigma_1=
\int_\tau^\infty f(\sigma_1,\sigma_1)g(\sigma_1,\sigma_1)d\sigma_1 \\
&-\int_\tau^\infty f(\tau,\sigma_1)g(\tau,\sigma_1)d\sigma_1
-\int_\tau^\infty \int_{\sigma_1}^\infty f(\sigma_1,\sigma_2)
\partial_{\sigma_1}g(\sigma_1,\sigma_2)d\sigma_2 d\sigma_1
\end{align*}
to conclude that the operator $\mcA_+$ decomposes into four types of terms,
$\mcA_+=\sum_{j=1}^4 \mcA_j$, of the form
\begin{align*}
\mcA_1 x(\tau,\xi)&=\int_\tau^\infty
\int_0^\infty \tilde K_{n_0}^{nd}\left (\omega(\tau,\sigma_1)^2 \xi,
\eta \right ) \\
&\quad \times \int_{\sigma_1}^\infty
O(\sigma_1^{-\frac32}) \sigma_2^{-1}\omega(\tau, \sigma_2)
x(\sigma_2,\omega(\sigma_1,\sigma_2)^2\eta)d\sigma_2 d\eta d\sigma_1, \\
\mcA_2 x(\tau,\xi)&=\int_\tau^\infty
\int_0^\infty O(\sigma_1^{-\frac32})\omega(\tau, \sigma_1)
\tilde K_{n_0}^{nd}\left (\omega(\tau, \sigma_1)^2 \xi,
\eta \right )
x(\sigma_1,\eta)d\eta d\sigma_1
\end{align*}
as well as
\begin{align*}
\mcA_3x(\tau,\xi)&=\int_0^\infty \tilde K_{n_0}^{nd}(\xi,\eta)
\int_\tau^\infty O(\tau^{-1})\sigma_1^{-\frac12}\omega(\tau, \sigma_1)
x(\sigma_1,\omega(\tau, \sigma_1)^2\eta)d\sigma_1 d\eta, \\
\mcA_4 x(\tau,\xi)&=
\int_\tau^\infty
\int_0^\infty \partial_{\sigma_1}
\tilde K_{n_0}^{nd}\left (\omega(\tau, \sigma_1)^2 \xi,
\omega(\sigma_1, \sigma_2)^{-2}\eta \right ) \\
&\quad \times \int_{\sigma_1}^\infty
O(\sigma_1^{-\frac12}) \sigma_2^{-1}\omega(\tau, \sigma_2)
\omega(\sigma_1, \sigma_2)^{-2}
x(\sigma_2,\eta)d\sigma_2 d\eta d\sigma_1
\end{align*}
with
\[ \tilde K_{n_0}^{nd}(\xi,\eta):=\xi^\frac12 \frac{K_{n_0}^{nd}(\xi,\eta)}{\xi^\frac12-\eta^\frac12} \]
where we have used the fact that $\sigma_1^{-\frac12}\lesssim \xi^\frac12$ on the support
of the cut-off $1-\chi(\sigma_1 \xi)$ and performed the change of variable
$\eta \mapsto \omega(\sigma_1, \sigma_2)^2\eta$ in the first three terms.
Since
\[ \left |\frac{\eta^\frac12}{\xi^\frac12}-1\right |\gtrsim \frac{1}{n_0} \]
on the support of $K_{n_0}^{nd}$ (cf.~the proof of Lemma \ref{lem:Kccnd}), we observe
that
\[ |\tilde K_{n_0}^{nd}(\xi,\eta)|\lesssim n_0 |K_{n_0}^{nd}(\xi,\eta)| \] and thus,
Lemmas \ref{lem:estB} and \ref{lem:Kccnd} yield
$\|\mcA_j x\|_{\mcY^\alpha}\lesssim n_0^3 \tau_0^{-\frac12}\|x\|_{\mcY^\alpha}$
for $j=1,2,3$.
Finally, after the change of variables $\eta\mapsto \omega(\sigma_1, \sigma_2)^2\eta$,
the operator $\mcA_4$ can be written as
\begin{align*}
\mcA_4x(\tau,\xi)&=
\int_\tau^\infty
\int_0^\infty
[\xi \partial_\xi+\eta \partial_\eta] \tilde K_{n_0}^{nd}\left (\omega(\tau, \sigma_1)^2 \xi,\eta \right ) \\
&\quad \times \int_{\sigma_1}^\infty
O(\sigma_1^{-\frac32}) \sigma_2^{-1} \omega(\tau, \sigma_2)
x(\sigma_2,\omega(\sigma_1, \sigma_2)^2\eta)d\sigma_2 d\eta d\sigma_1
\end{align*}
and since $|[\xi \partial_\xi+\eta \partial_\eta]\tilde K_{n_0}^{nd}(\xi,\eta)|
\lesssim n_0^2 |K_{n_0}^{nd}(\xi,\eta)|$ by Theorem \ref{thm:spec}, we obtain
\[ \|\mcA_4 x\|_{\mcY^\alpha}\lesssim n_0^4 \tau_0^{-\frac12}\|x\|_{\mcY^\alpha} \]
as before by Lemmas \ref{lem:estB} and \ref{lem:Kccnd}.
\end{proof}

\subsection{Estimates for the diagonal term}
Next, we consider the diagonal operator $\mcK_{n_0}^d$.
We further decompose $\mcK_{n_0}^d$ in the following way.
We set
\begin{align*}
K_1^\epsilon(\xi,\eta)&:=1_{[0,\epsilon)}(\xi)K_{n_0}^d(\xi,\eta) \\
K_3^\epsilon(\xi,\eta)&:=1_{(\epsilon^{-1},\infty)}(\xi)K_{n_0}^d(\xi,\eta)
\end{align*}
where $K_{n_0}^d$ is the kernel of $\mcK_{n_0}^d$.
Following our usual scheme, we denote the operator with kernel $K_j^{\epsilon}$ by
$\mcK_j^\epsilon$,
$j=1,3$. Furthermore, we define $\mcK_2^\epsilon$ by
\[ \mcK_{n_0}^d=\sum_{j=1}^3 \mcK_j^\epsilon \]
which yields the desired decomposition.
We bear in mind that the operators $\mcK_j^\epsilon$ depend on $n_0$ but suppress this
dependence in the notation.
Finally, we set
\begin{align*}
\mcA_\epsilon&:=2\hat{\mcD}_c \mcH_c \beta \mcK_1^\epsilon \\
\mcB_\epsilon&:=2\hat{\mcD}_c \mcH_c \beta \mcK_2^\epsilon \\
\mcC_\epsilon&:=2\hat{\mcD}_c \mcH_c \beta \mcK_3^\epsilon.
\end{align*}
First, we establish some smallness properties.
Here and in the following, the product of noncommutative operators $A_j$ is defined as
\[ \prod_{j=1}^n A_j:=A_1 A_2 \dots A_n. \]

\begin{lem}
\label{lem:ABC}
Let $\alpha> \frac34(1+\tfrac{1}{\nu})$ and $\epsilon>0$ be sufficiently small.
Then we have the bounds
\[ \|\mcA_\epsilon\|_{\mcY^\alpha}\lesssim \epsilon^\frac14,
\quad \|\mcC_\epsilon\|_{\mcY^\alpha}\lesssim \epsilon^\frac14 \]
as well as
\[ \left \|\prod_{j=0}^{n_0-1}\mcB_{\mu^{-j}\epsilon}\right \|_{\mcY^\alpha}
\leq \left [\frac{C^{n_0-1}\epsilon^{-(n_0-1)}}{(n_0-1)!} \right ]^\frac{1}{p}
\]
for all sufficiently large $n_0 \in \N$
where $\mu:=1+\frac{4}{n_0}$, $C>0$ is some absolute constant
and $p$ is from Definition \ref{def:spaces}.
\end{lem}

\begin{proof}
From Lemma \ref{lem:Kccd} we immediately obtain the estimate
\begin{align*}
\|\mcK_1^\epsilon f\|_Y&=\|1_{[0,\epsilon)}|\cdot|^\frac14 |\cdot|^{-\frac14}\mcK_{n_0}^d f\|_Y
\lesssim \|1_{[0,\epsilon)}|\cdot|^\frac14\|_{L^\infty(0,\infty)}\|f\|_Y \\
&=\epsilon^\frac14 \|f\|_Y
\end{align*}
and analogously, $\|\mcK_3^\epsilon\|_Y\lesssim \epsilon^\frac14$, uniformly in $n_0\geq 100$.
With $\omega(s):=s^{1+\frac{1}{\nu}}$ we have
\[ \mcA_\epsilon x(\tau,\sigma)=2\int_\tau^\infty \hat H_c(\tau,\sigma,\xi)\beta(\sigma)
\mcK_1^\epsilon x(\sigma,\omega(\tau, \sigma)^2\xi)d\sigma \]
and $|\hat H_c(\tau,\sigma,\xi)\beta(\sigma)|\lesssim \omega(\tau, \sigma)\sigma^{-1}$ (cf.~the
proof of Lemma \ref{lem:offdiag}).
Consequently, Lemma \ref{lem:estB} yields the stated estimate for $\mcA_\epsilon$ and
the proof for $\mcC_\epsilon$ is identical.

In order to prove the remaining estimate note first that
\begin{align}
\label{eq:Bn0}
\Big (&\prod_{j=0}^{n_0-1}\mcB_{\mu^{-j}\epsilon}\Big )x(\sigma_0,\eta_0)\nonumber \\
&=\int_{\sigma_0}^\infty \int_0^\infty \cdots \int_{\sigma_{n_0-1}}^\infty\int_0^\infty
x(\sigma_{n_0},\eta_{n_0})\prod_{j=0}^{n_0-1}\Big [2\hat H_c(\sigma_j,\sigma_{j+1},\eta_j)\beta(\sigma_{j+1})
\nonumber \\
&\quad \times K_2^{\mu^{-j}\epsilon}\left (\omega(\sigma_j,\sigma_{j+1})^2\eta_j,\eta_{j+1}
\right ) \Big ]d\eta_{n_0}d\sigma_{n_0}\cdots d\eta_1 d\sigma_1.
\end{align}
Now we are going to exploit the following observation.
Consider the expression
\begin{equation}
\label{eq:K2K2} K_2^\epsilon(\omega(\sigma_0, \sigma_1)^2\eta_0,\eta_1)K_2^{\mu^{-1}\epsilon}
(\omega(\sigma_1, \sigma_2)^2\eta_1,\eta_2)
\end{equation}
which appears in the integrand of \eqref{eq:Bn0}.
Assume $\sigma_0, \sigma_1$ to be fixed and suppose
we want to perform the integration with respect
to $\sigma_2$.
As $\sigma_2\to\infty$ we must have $\eta_1\to \infty$ in order to stay in the support of
\eqref{eq:K2K2}.
However, since $K_2^\epsilon$ is supported near the diagonal, this also entails
$\omega(\sigma_0, \sigma_1)^2\eta_0 \to \infty$ and we therefore necessarily leave
the support of \eqref{eq:K2K2}.
Hence, it is not necessary to integrate all the way up to infinity.
In order to quantify this argument we return to Eq.~\eqref{eq:Bn0} and note that on the support
of the integrand we have $\omega(\sigma_j, \sigma_{j+1})^2\frac{\eta_j}{\eta_{j+1}}
\geq 1-\frac{2}{n_0}$ for all $j\in \{0,1,\dots,n_0-1\}$. This implies
\begin{equation}
\label{eq:estprod}
 (1-\tfrac{2}{n_0})^{n_0-1}\leq \prod_{j=0}^{n_0-2}\omega(\sigma_j, \sigma_{j+1})^2
\tfrac{\eta_j}{\eta_{j+1}}=\omega(\sigma_0, \sigma_{n_0-1})^2 \tfrac{\eta_0}{\eta_{n_0-1}}.
\end{equation}
On the other hand, we have $\omega(\sigma_{n_0-1}, \sigma_{n_0})^2 \eta_{n_0-1}
\geq \mu^{-(n_0-1)}\epsilon$ and $\omega(\sigma_0, \sigma_1)^2 \eta_0\leq \epsilon^{-1}$
on the support of the integrand in Eq.~\eqref{eq:Bn0}.
By inserting these two estimates into \eqref{eq:estprod} we find
\[ (1-\tfrac{2}{n_0})^{n_0-1}\leq \omega(\tfrac{\sigma_1}{\sigma_{n_0}})^2 \mu^{n_0-1}\epsilon^{-2} \]
which yields $\sigma_{n_0}^{2(1+\frac{1}{\nu})}\lesssim \sigma_1^{2(1+\frac{1}{\nu})}\epsilon^{-2}$ for
all $n_0\geq 100$.
Hence, we obtain the crude bound $\sigma_{n_0}\leq \sigma_1 \epsilon^{-1}$ for all
$n_0\geq 100$ (provided $\epsilon>0$ is sufficiently small) and, since $\sigma_0 \leq \sigma_1
\leq \dots \leq \sigma_{n_0}$, we have in fact $\sigma_j \leq \sigma_1\epsilon^{-1}$
for all $j \in \{2,3,\dots,n_0\}$.
Consequently, upon writing
\[ \mcJ_a x(\tau,\xi):=2\int_\tau^{a}\hat H_c(\tau,\sigma,\xi)
\beta(\sigma)x(\sigma, \omega(\tau, \sigma)^2\xi)d\sigma \]
and by defining
\[ y(\tau,\xi):=\prod_{j=1}^{n_0-1}
\left ( \mcJ_{\epsilon^{-1}\tau} \mcK_2^{\mu^{-j}\epsilon}\right )x(\tau,\xi), \]
we obtain
\[ \Big (\prod_{j=0}^{n_0-1}\mcB_{\mu^{-j}\epsilon}\Big ) x(\tau,\xi)
=\mcB_\epsilon y(\tau,\xi). \]
Therefore, it suffices to prove an appropriate bound for $y(\tau,\xi)$.
By H\"older's inequality, Fubini, and the fact that
$|\hat H_c(\tau,\sigma,\xi)|\lesssim \omega_{\nu}(\tfrac{\tau}{\sigma})$,
we infer
\begin{equation}
\label{eq:estLp}
\|\mcJ_{\epsilon^{-1}\tau} x(\tau,\cdot)\|_{L^p}^p\lesssim \tau^{-1}
\int_\tau^{\epsilon^{-1}\tau}
\omega_{\nu}(\tfrac{\tau}{\sigma})^{p-2}\|x(\sigma,\cdot)\|_{L^p}^p d\sigma
\end{equation}
and similarly,
\begin{equation}
\label{eq:estL2}
\|\mcJ_{\epsilon^{-1}\tau}x(\tau,\cdot)\langle \cdot\rangle^\frac18 \|_{L^2_\rho}^2
\lesssim \tau^{-\delta_0}\int_\tau^{\epsilon^{-1}\tau} \omega_{\nu}(\tfrac{\tau}{\sigma})^{2-\frac72}
\sigma^{-1+\delta_0}
\|x(\sigma,\cdot)\langle \cdot \rangle^\frac18\|_{L^2_\rho}^2 d\sigma
\end{equation}
where $\delta_0>0$ is chosen such that $\alpha\geq \frac34(1+\frac{1}{\nu})+\delta_0$, i.e.,
such that the integral converges (cf.~the proof of Lemma \ref{lem:estB}).
By Lemma \ref{lem:Kccd} we have the same bounds for $\mcJ_{\epsilon^{-1}\tau}\mcK_2^{\mu^{-j}\epsilon}$,
$j \in \{0,\dots,n_0-1\}$.
Consequently, \eqref{eq:estLp} implies
\begin{align*}
 \|y(\sigma_1,\cdot)\|_{L^p}^p&\leq C^{n_0-1}\int_{\sigma_1}^{\epsilon^{-1}\sigma_1}
\int_{\sigma_2}^{\epsilon^{-1}\sigma_1}\cdots \int_{\sigma_{n_0-1}}^{\epsilon^{-1}\sigma_1}
\prod_{j=1}^{n_0-1}\left (\omega_{\nu}(\tfrac{\sigma_j}{\sigma_{j+1}})^{p-2}\sigma_j^{-1}\right ) \\
&\quad \times \|x(\sigma_{n_0},\cdot)\|_{L^p}^p d\sigma_{n_0} d\sigma_{n_0-1}\dots d\sigma_2 \\
&\leq \frac{C^{n_0-1}\epsilon^{-(n_0-1)}}{(n_0-1)!}\sigma_1^{-p\alpha}\|x\|_{\mcY^\alpha}^p
\end{align*}
where $C>0$ is some absolute constant.
By the same argument one obtains an analogous estimate for $\|y(\sigma_1,\cdot)\|_{L^2_\rho}$
and the proof is finished.
\end{proof}

Next, we prove the following crucial orthogonality relations.

\begin{lem}
\label{lem:orth}
For any sufficiently small $\epsilon>0$ and $n_0\geq 4$ we have
\[ \mcA_\epsilon \mcB_{\mu\epsilon}=0, \quad
\mcB_{\mu\epsilon} \mcC_\epsilon=0 \]
as well as
\[ \mcA_\epsilon \mcC_\epsilon=\mcA_{\mu\epsilon} \mcC_\epsilon
= \mcA_{\epsilon} \mcC_{\mu\epsilon}=0 \]
where $\mu:=1+\frac{4}{n_0}$.
\end{lem}

\begin{proof}
Explicitly, we have
\[ \mcB_\epsilon x(\tau,\xi)=2\int_\tau^\infty \hat H_c(\tau,\sigma,\xi)\beta(\sigma)
\int_0^\infty K_2^\epsilon(\omega(\tau, \sigma)^2 \xi,\eta_2)x(\sigma,\eta_2)d\eta_2 d\sigma \]
where, as before, $\omega(s_{1}, s_2)=\kappa(s_1)\kappa^{-1}(s_2)$.
Furthermore,
\begin{align*}
\mcK_1^\epsilon \mcB_{\mu\epsilon}x(\tau,\xi)&=
\int_0^\infty K_1^\epsilon(\xi,\eta_1)\mcB_{\mu\epsilon}x(\tau,\eta_1)d\eta_1
\end{align*}
and thus, in order to prove $\mcA_\epsilon \mcB_{\mu\epsilon}=0$ it suffices to show that
\begin{equation}
\label{eq:K1K20}
K_1^\epsilon(\xi,\eta_1)K_{2}^{\mu\epsilon}
(\omega(\tau, \sigma)^2 \eta_1,\eta_2)=0
\end{equation}
for all $\xi,\eta_1,\eta_2\geq 0$ and $\tau \leq \sigma$.
Recall that
\[ K_{1}^\epsilon(\xi,\eta_1)=1_{[0,\epsilon)}(\xi)\chi(n_0(\tfrac{\xi}{\eta_1}-1))
K_{cc}(\xi,\eta_1) \]
where $\chi(x)\not=0$ only if $|x|\leq 2$
and thus, $K_{1}^\epsilon(\xi,\eta_1)\not=0$ only if
\begin{equation}
\label{eq:condeta11}
\eta_1\leq (1-\tfrac{2}{n_0})^{-1}\xi<
(1-\tfrac{2}{n_0})^{-1}\epsilon.
\end{equation}
On the other hand, we have
$K_{2}^{\mu\epsilon}(\omega(\tau, \sigma)^2\eta_1,\eta_2)\not=0$
only if $(1+\frac{4}{n_0})\epsilon\leq \omega(\tau, \sigma)^2\eta_1\leq \eta_1$
and since $1+\frac{4}{n_0}\geq (1-\frac{2}{n_0})^{-1}$ for all $n_0\geq 4$, this condition is incompatible
with \eqref{eq:condeta11} which proves \eqref{eq:K1K20}.

Similarly, to see that $\mcB_{\mu\epsilon} \mcC_\epsilon=0$, we consider
the product kernel
\[ K_{2}^{\mu\epsilon}(\xi,\eta_1)K_{3}^\epsilon(\omega(\tau, \sigma)^2\eta_1,\eta_2). \]
The second factor is non-vanishing only if $\eta_1>\epsilon^{-1}$ whereas on the support of the first factor
we have
\[ \eta_1\leq (1-\tfrac{2}{n_0})^{-1}\xi\leq (1-\tfrac{2}{n_0})^{-1}(1+\tfrac{4}{n_0})^{-1}\epsilon^{-1} \]
and $(1-\frac{2}{n_0})^{-1}(1+\frac{4}{n_0})^{-1}\leq 1$ for all $n_0\geq 4$. This implies
the desired $\mcB_{\mu\epsilon} \mcC_\epsilon=0$.
The remaining assertions are immediate provided $\epsilon>0$ is sufficiently small.
\end{proof}

Now we can show that $(2\hat{\mcD}_c \mcH_c \beta \mcK_{n_0}^d)^{2n_0}$ has small
operator norm on $\mcY^\alpha$ provided $n_0$ is sufficiently large.

\begin{lem}
\label{lem:diag}
Let $\alpha>\frac34(1+\frac{1}{\nu})$ and $\delta_0>0$. Then
\[ \|(2\hat{\mcD}_c \mcH_c \beta \mcK_{n_0}^d)^{2n_0}\|_{\mcY^\alpha}<\delta_0 \]
provided $n_0 \in \N$ is sufficiently large.
\end{lem}

\begin{proof}
For brevity we write $\mcJ:=2\hat {\mcD}_c \mcH_c \beta \mcK_{n_0}^d$.
We have
\begin{align*}
 \mcJ^{2n_0}&=(\mcA_\epsilon+\mcB_\epsilon+\mcC_\epsilon)\mcJ^{2n_0-1} \\
 &=\mcA_\epsilon \mcJ^{2n_0-1}+\mcB_\epsilon \mcJ^{2n_0-1}+\mcC_\epsilon \mcJ^{2n_0-1}
 \end{align*}
 and consider each term separately.
 Furthermore, we set $\mu:=1+\frac{4}{n_0}$.
 For the first term we note that
 \[ \mcA_\epsilon \mcJ^{2n_0-1}=\mcA_\epsilon ( \mcA_{\mu\epsilon}+\mcB_{\mu \epsilon}
 +\mcC_{\mu \epsilon})\mcJ^{n_0-2}=\mcA_\epsilon \mcA_{\mu \epsilon}\mcJ^{2n_0-2} \]
 by Lemma \ref{lem:orth}.
 Thus, inductively we find
 \begin{equation}
 \label{eq:prodA}
  \mcA_\epsilon \mcJ^{2n_0-1}=\prod_{j=0}^{2n_0-1}\mcA_{\mu^j \epsilon}
  \end{equation}
 and Lemma \ref{lem:ABC} yields
 \begin{equation*}
  \|\mcA_\epsilon \mcJ^{2n_0-1}\|_{\mcY^\alpha}
 \leq \left (C\mu^{\frac{2n_0}{4}}\epsilon^{\frac14} \right )^{2n_0}.
 \end{equation*}
 For the second term we obtain
 \begin{align*} \mcB_\epsilon \mcJ^{2n_0-1}&=\mcB_\epsilon (\mcA_{\mu^{-1}\epsilon}
 +\mcB_{\mu^{-1} \epsilon}
 +\mcC_{\mu^{-1} \epsilon})\mcJ^{n_0-2} \\
 &=\mcB_\epsilon \mcA_{\mu^{-1}\epsilon} \mcJ^{2n_0-2}+
 \mcB_\epsilon \mcB_{\mu^{-1}\epsilon}
\mcJ^{2n_0-2}  \\
&=\mcB_\epsilon \prod_{j=0}^{2n_0-2}\mcA_{\mu^{-1+j}\epsilon}
+\mcB_\epsilon \mcB_{\mu^{-1}\epsilon}\mcJ^{2n_0-2}
 \end{align*}
 by Lemma \ref{lem:orth} and Eq.~\eqref{eq:prodA}.
 Inductively we see that this is a sum of $2n_0$ terms which consist of products of consecutive
 $\mcB$'s and consecutive $\mcA$'s.
 We thus may write $\mcB_\epsilon \mcJ^{2n_0-1}=\mcS_1+\mcS_2$ where $\mcS_1$
 contains all the terms with at most $n_0$ $\mcB$'s.
From Lemma \ref{lem:ABC} we obtain the bounds
 \begin{align*}
 \|\mcS_1\|_{\mcY^\alpha}\leq \left (C\mu^{\frac{n_0}{4}}\epsilon^\frac14 \right )^{n_0},
 \quad \|\mcS_2\|_{\mcY^\alpha}\leq \left [\frac{C^{n_0}\epsilon^{-n_0}}{n_0!} \right ]^\frac{1}{p}
 \end{align*}
 provided $n_0$ is sufficiently large.
 Finally, we have
 \[ \mcC_\epsilon \mcJ^{2n_0-1}=\mcC_\epsilon \mcA_\epsilon \mcJ^{2n_0-2}
 +\mcC_\epsilon \mcB_\epsilon \mcJ^{2n_0-2}+\mcC_\epsilon^2 \mcJ^{2n_0-2} \]
and thus, by the exact same token as before we obtain a decomposition
$C_\epsilon \mcJ^{2n_0-1}=\mcS_3+\mcS_4$ with the bounds
\[ \|\mcS_3\|_{\mcY^\alpha}\leq \left (C\mu^{\frac{n_0}{4}}\epsilon^\frac14 \right )^{n_0},
 \quad \|\mcS_4\|_{\mcY^\alpha}\leq \left [\frac{C^{n_0}\epsilon^{-n_0}}{n_0!} \right ]^\frac{1}{p}
 \]
 for sufficiently large $n_0$.
 Hence, by first choosing $\epsilon>0$ sufficiently small and then $n_0$ sufficiently large,
 the claim follows.
 \end{proof}

Now we can conclude the existence of $(1+2\hat{\mcD}_c \mcH_c \beta \mcK_{cc})^{-1}$.

\begin{cor}
\label{cor:invDHKc}
If $\alpha>\frac34(1+\frac{1}{\nu})$ then the
operator $1+2\hat{\mcD}_c \mcH_c \beta \mcK_{cc}$ has a bounded inverse on $\mcY^\alpha$.
\end{cor}

\begin{proof}
For brevity we write $\mcJ:=2\hat{\mcD}_c \mcH_c \beta \mcK_{cc}$ and decompose
\begin{align*} \mcJ^{2n_0}&=
(2\hat{\mcD}_c \mcH_c \beta \mcK_{n_0}^{nd}+2\hat{\mcD}_c \mcH_c \beta \mcK_{n_0}^d)^{2n_0}  \\
&=\mcS+(2\hat{\mcD}_c \mcH_c \beta \mcK_{n_0}^d)^{2n_0}
\end{align*}
where $\mcS$ consists of $2^{2n_0}-1$ terms, each of which containing the operator
$\hat{\mcD}_c \mcH_c \beta \mcK_{n_0}^{nd}\hat{\mcD}_c \mcH_c \beta$.
Hence, by first choosing $n_0$ sufficiently large and then $\tau_0$ sufficiently large
(depending on $n_0$), we obtain from Lemmas \ref{lem:offdiag} and \ref{lem:diag} the bound
\[ \|\mcJ^{2n_0}\|_{\mcY^\alpha}<1. \]
This implies
\[ \sum_{n=0}^\infty \|\mcJ^n\|_{\mcY^\alpha}=\sum_{k=0}^\infty \sum_{\ell=0}^{2n_0-1}
\|\mcJ^{2n_0k+\ell}\|_{\mcY^\alpha}\lesssim \sum_{k=0}^\infty \|\mcJ^{2n_0}\|_{\mcY^\alpha}^k
\lesssim 1 \]
and the claim follows.
\end{proof}

\subsection{The inverse of $1+2\hat{\mcD} \mcH \beta \mcK$}

Finally, we consider the matrix operator $1+2\hat{\mcD} \mcH \beta \mcK$ which explicitly
reads
\[ 1+2\hat{\mcD} \mcH \beta \mcK=\left (\begin{array}{cc}
1+2\hat{\mcD}_d \mcH_d \beta \mcK_{dd} & 2\hat{\mcD}_d \mcH_d \beta \mcK_{dc} \\
2\hat{\mcD}_c \mcH_c \beta \mcK_{cd} & 1+2 \hat{\mcD}_c \mcH_c \beta \mcK_{cc}
\end{array} \right ), \]
cf.~Eq.~\eqref{eq:matrixDHK}.

\begin{lem}
\label{lem:invmatrixDHK}
Let $\alpha_c>\frac34(1+\frac{1}{\nu})$ and $\alpha_c\leq \alpha_d <\alpha_c+1$.
Then the operator $1+2\hat{\mcD}\mcH \beta \mcK$ has a bounded
inverse on $\mcY^{\alpha_d,\alpha_c}$.
\end{lem}

\begin{proof}
For brevity we write
\begin{align*}
\mcJ_{dd}&:=1+2\hat{\mcD}_d \mcH_d \beta \mcK_{dd}&
\mcJ_{dc}&:=2\hat{\mcD}_d \mcH_d \beta \mcK_{dc} \\
\mcJ_{cd}&:=2\hat{\mcD}_c \mcH_c \beta \mcK_{cd} &\mcJ_{cc}&:=
1+2 \hat{\mcD}_c \mcH_c \beta \mcK_{cc}.
\end{align*}
From Corollary \ref{cor:invDHKc} and the comment following Eq.~\eqref{eq:matrixDHK}
we know that $\mathrm{diag}(\mcJ_{dd},\mcJ_{cc})^{-1}=\mathrm{diag}(\mcJ_{dd}^{-1}, \mcJ_{cc}^{-1})$
exists as a bounded operator
on $\mcY^{\alpha_d,\alpha_c}$.
Consequently, the equation
\[ \left (\begin{array}{cc} \mcJ_{dd} & \mcJ_{dc} \\
\mcJ_{cd} & \mcJ_{cc} \end{array} \right )\left (\begin{array}{c} x_d \\ x \end{array}\right)
=\left (\begin{array}{c} y_d \\ y \end{array} \right ) \]
implies
\begin{align*}
(1-\mcJ_{dd}^{-1}\mcJ_{dc} \mcJ_{cc}^{-1} \mcJ_{cd})x_d&=\mcJ_{dd}^{-1}(y_d-\mcJ_{dc}
\mcJ_{cc}^{-1} y) \\
(1-\mcJ_{cc}^{-1}\mcJ_{cd}\mcJ_{dd}^{-1}\mcJ_{dc})x&=\mcJ_{cc}^{-1}(y-\mcJ_{cd}\mcJ_{dd}^{-1}y_d)
\end{align*}
and it suffices to prove smallness of $\mcJ_{dc}$.
From Theorem \ref{thm:K} and Lemma \ref{lem:estH} we obtain
\[ |\mcJ_{dc}x(\tau)|=2|\hat{\mcD}_d \mcH_d \beta \mcK_{dc} x(\tau)|
\lesssim \tau^{-\alpha_c-1}\sup_{\tau>\tau_0}\tau^{\alpha_c} \|x(\tau,\cdot)\|_Y \]
which yields
\[ \sup_{\tau>\tau_0}\tau^{\alpha_d}|\mcJ_{dc}x(\tau)|\lesssim \tau_0^{\alpha_d-\alpha_c-1}
\|x\|_{\mcY^{\alpha_c}} \]
and we obtain smallness by choosing $\tau_0$ sufficiently large since
$\alpha_d-\alpha_c-1<0$ by assumption.
\end{proof}

\subsection{The inverse of $1+2\mcH\beta \mcK \hat{\mcD}$}

For the following it is also necessary to invert the operator $1+2\mcH\beta \mcK \hat{\mcD}$.
As before, we first consider the difficult continuous component $\mcH_c \beta \mcK_{cc}\hat{\mcD}_c$
which is explicitly given by
\begin{align}
\label{eq:HbKD}
\mcH_c\beta\mcK_{cc}\hat {\mcD}_c x(\tau,\xi)&=\int_\tau^\infty
H_c(\tau,\sigma,\xi)\beta(\sigma)\mcK_{cc}\hat{\mcD}_c x(\sigma,\omega(\tau, \sigma)\xi)d\sigma
\nonumber \\
&=\int_\tau^\infty H_c(\tau,\sigma,\xi)\beta(\sigma)
\int_0^\infty K_{cc}(\omega(\tau, \sigma)^2 \xi,\eta)\hat {\mcD}_c x(\sigma,\eta)d\eta d\sigma
\end{align}
with $\omega(s_1, s_2)=\kappa(s_1)\kappa^{-1}(s_2)$.
Now recall that
\[ \hat{\mcD}_c x(\tau,\xi)=\partial_1 x(\tau,\xi)+\beta(\tau) \left [-2\xi \partial_2 x(\tau,\xi)
-(\tfrac52+\tfrac{\xi\rho'(\xi)}{\rho(\xi)})x(\tau,\xi) \right ]\]
and
\[ \hat {\mcD}_c x(\tau,\xi)=\kappa(\tau)^\frac52 \rho(\xi)^{-\frac12}[\partial_\tau-2\beta(\tau)\xi
\partial_\xi]\left [\kappa(\tau)^{-\frac52}\rho(\xi)^\frac12 x(\tau,\xi)\right ] \]
where
\[  \beta(\tau)=\tfrac{\kappa'(\tau)}{\kappa(\tau)}
\sim \tau^{-1}. \]
Hence, if we set $y(\tau,\xi):=\kappa(\tau)^{-\frac52}\rho(\xi)x(\tau,\xi)$ we may write
\begin{equation}
\label{eq:conj}
 \hat {\mcD}_c x(\tau,\xi)=\kappa(\tau)^\frac52 \rho(\xi)^{-\frac12}\tilde{\mcD}_c
y(\tau,\xi)
\end{equation}
where $\tilde {\mcD}_c y(\tau,\xi):=\partial_1 y(\tau,\xi)-2\beta(\tau)\partial_2 y(\tau,\xi)$.
Now observe that
\[ \tilde{\mcD}_c y(\sigma,\omega(\tau, \sigma)^2\eta)=
\partial_\sigma y(\sigma,\omega(\tau, \sigma)^2\eta) \]
since $\partial_\sigma \omega(\tau, \sigma)^2=-2\beta(\sigma)\omega(\tau, \sigma)^2$.
Thus, from the definition of $y$ and Eq.~\eqref{eq:conj} we infer the identity
\begin{align}
\label{eq:Dcident}
\hat{\mcD}_c x(\sigma,\omega(\tau, \sigma)^2\eta)
&=\kappa(\sigma)^\frac52 \rho\left (\omega(\tau, \sigma)^2\eta\right)^{-\frac12} \nonumber \\
&\quad \times \partial_\sigma \left [ \kappa(\sigma)^{-\frac52}\rho\left (
\omega(\tau, \sigma)^2\eta\right )^\frac12 x(\sigma,\omega(\tau, \sigma)^2\eta) \right ].
\end{align}
Consequently, after the change of variable $\eta\mapsto \omega(\tau, \sigma)^2\eta$,
Eq.~\eqref{eq:HbKD} can be rewritten as
\begin{align}
\label{eq:HbKD2}
\mcH_c&\beta\mcK_{cc}\hat {\mcD}_c x(\tau,\xi)=\int_\tau^\infty \int_0^\infty
H_c(\tau,\sigma,\xi)\beta(\sigma)\omega(\tau, \sigma)^2
K_{cc}(\omega(\tau, \sigma)^2\xi,\omega(\tau, \sigma)^2\eta) \nonumber \\
& \times \kappa(\sigma)^\frac52 \rho\left (\omega(\tau, \sigma)^2\eta\right)^{-\frac12}
\partial_\sigma \left [ \kappa(\sigma)^{-\frac52}\rho\left (
\omega(\tau, \sigma)^2\eta\right )^\frac12 x(\sigma,\omega(\tau, \sigma)^2\eta) \right ]
d\eta d\sigma.
\end{align}

\begin{lem}
\label{lem:invHKD}
Let $\alpha> 1+\frac34(1+\frac{1}{\nu})$. Then the operator
$1+2\mcH_c \beta \mcK_{cc}\hat{\mcD}_c$ has a bounded inverse on $\mcX^\alpha$.
\end{lem}

\begin{proof}
Performing an integration by parts with respect to $\sigma$ in Eq.~\eqref{eq:HbKD2} and noting
that
\[ \partial_\sigma \left [\xi^{-\frac12}
\sin\left (\xi^{\frac{1}{2}}\kappa(\tau)\int_{\tau}^{\sigma}\kappa^{-1}(u)\,du\right ) \right ]
=\omega(\tau, \sigma)\cos\left (\xi^{\frac{1}{2}}\kappa(\tau)\int_{\tau}^{\sigma}\kappa^{-1}(u)\,du\right )
\]
we obtain the decomposition $\mcH_c \beta \mcK_{cc} \hat{\mcD}_c=\mcA_1+\mcA_2$ where
\[ \mcA_1 x(\tau,\xi)=-\int_\tau^\infty \omega(\tau, \sigma)\hat H_c(\tau,\sigma,\xi)\beta(\sigma)
\mcK_{cc}x(\sigma,\omega(\tau, \sigma)^2\xi)d\sigma \]
and
\[ \mcA_2 x(\tau,\xi)=\int_\tau^\infty H_c(\tau,\sigma,\xi)O(\sigma^{-2})
\tilde{\mcK}_{cc}x(\sigma,\omega(\tau, \sigma)^2\xi)d\sigma. \]
The kernel of the operator $\tilde{\mcK}_{cc}$ consists of a linear combination of
$\sigma$-derivatives of $K_{cc}(\omega(\tau, \sigma)^2\xi,\omega(\tau, \sigma)^2\eta)$
and is therefore of the same type as the kernel of the commutator $[\mcA_c,\mcK_{cc}]$
(cf.~the proof of Proposition \ref{prop:KAKc}).
In particular, $\tilde{\mcK}_{cc}$ maps the space $X$ to $Y$ (Proposition \ref{prop:KAKc})
and Lemma \ref{lem:estH} immediately yields
\[ \|\mcA_2 x\|_{\mcX^\alpha}\lesssim \nu^{2\delta}\tau_0^{-1+2\delta}\|x\|_{\mcX^\alpha}. \]
Consequently, smallness can be achieved by choosing $\tau_0$ sufficiently large (depending on $\nu$).
The operator $\mcA_1$, on the other hand, is of the same type as $\hat{\mcD}_c \mcH_c \beta \mcK_{cc}$ but
this time viewed as a map from $\mcX^\alpha$ to $\mcX^\alpha$.
However, this does not make any difference since the crucial Lemmas \ref{lem:estB},
\ref{lem:Kccnd} and \ref{lem:Kccd} are valid for the space $\mcX^\alpha$ (respectively
$X$) as well.
Note carefully that the stronger requirement $\alpha>1+\frac{11}{4}(1+\frac{1}{\nu})-\gamma$
in Lemma \ref{lem:estB} is exactly compensated by the additional factor
$\omega(\tau, \sigma)$ in $\mcA_1$.
Hence, the operator $\mcA_1$ can be treated in the exact same fashion as
$\hat{\mcD}_c\mcH_c \beta \mcK_{cc}$ and in particular, Lemmas \ref{lem:offdiag} and
\ref{lem:diag} hold accordingly for $\mcA_1$ on $\mcX^\alpha$.
\end{proof}

As in the case of Lemma \ref{lem:invmatrixDHK}, Lemma \ref{lem:invHKD} implies
the invertibility of the matrix operator $1+2\mcH \beta \mcK \hat{\mcD}$.

\begin{lem}
\label{lem:invmatrixHKD}
Let $\alpha_c>\frac34(1+\frac{1}{\nu})$ and $\alpha_c\leq \alpha_d <\alpha_c+1$.
Then the operator $1+2\mcH \beta \mcK\hat{\mcD}$ has a bounded
inverse on $\mcX^{\alpha_d,\alpha_c}$.
\end{lem}

\begin{proof}
The proof is completely analogous to the proof of Lemma \ref{lem:invmatrixDHK} and
therefore omitted.
\end{proof}

\subsection{Solution of the main equation}

Now we are ready to solve the main equation \eqref{eq:transport}.
By setting $\mcN:=\sum_{j=1}^5 \mcN_j$ and $\hat{\mcK}:=\mcK^2+[\mcA,\mcK]+\mcK
+\frac{\beta'}{\beta^2}\mcK$ we rewrite Eq.~\eqref{eq:transport} as
\begin{equation}
\label{eq:sysFourier2}
[\hat {\mcD}^2+\beta \hat{\mcD}+\underline{\xi}]\mbx=\mbe+\mcN(\mbx)
-2\beta \mcK \hat{\mcD}\mbx-\beta^2 \hat{\mcK}\mbx
\end{equation}
where $\mbe(\tau,\xi)=\kappa(\tau)^{-2}\mathcal{F} [|\cdot|\tilde \chi(\tau,\cdot)e_2(\tau,\cdot)](\xi)$.
By applying $\mcH$ we find that Eq.~\eqref{eq:transport} is
equivalent to
\begin{equation}
\label{eq:sysFourier3}
\mbx
=\underline{\Phi}(\mbx):=
(1+2\mcH \beta \mcK \hat{\mcD})^{-1}\mcH
\left [ \mbe+\mcN(\mbx)-\beta^2 \hat{\mcK}\mbx \right ]
\end{equation}
We claim that Eq.~\eqref{eq:sysFourier3} has a solution
$(\mbx,\hat{\mcD}\mbx)\in \mcX^{\frac43-2\delta,\frac43-2\delta}\times \mcY^{\frac43,\frac43}$ with
$\delta$ from Definition \ref{def:spaces}.
In order to prove this we have to recall some mapping properties from \cite{DonKri12}.

\begin{prop}
\label{prop:N}
We have the estimate
\[ \|\mcN(\mbx)-\mcN(\underline{y})\|_{\mcY^{\frac43+\frac54,\frac43+\frac54}}\lesssim
\|\mbx-\underline{y}\|_{\mcX^{\frac43-2\delta,\frac43-2\delta}} \]
for all $\mbx,\underline{y}$ in the unit ball in $\mcX^{\frac43-2\delta,\frac43-2\delta}$.
\end{prop}

\begin{proof}
This follows by inspection of the proofs of the corresponding results in \cite{DonKri12}, in particular
Lemmas 4.4 -- 4.10. Note also that the loss of $\tau^\frac14$ discussed in Remark 4.8 in \cite{DonKri12}
does not occur
in our case since $\kappa(\tau)^{-2}$ is bounded for all $\tau\geq \tau_0$ (unlike
the corresponding $\tilde \lambda(\tau)^{-2}$ in \cite{DonKri12}).
In fact, as in \cite{DonKri12} one proves that the nonlinearity maps the space $X$
to $Y$ and for the time decay one gains at least $\tau^{-(\frac43-2\delta)}$ (from the quadratic
part $\mcN_2$, all other contributions are even better; the linear part $\mcN_1$ yields
a gain of $\tau^{-2}$, cf.~Lemma 4.5 in \cite{DonKri12}).
Since $\delta>0$ is assumed to be small, the stated bound follows.
\end{proof}

Another result from \cite{DonKri12} which we require concerns the mapping properties
of $\mcK$ and $\hat {\mcK}$.

\begin{prop}
\label{prop:K}
The operators $\mcK$ and $\hat{\mcK}$ satisfy the bounds
\begin{align*} \|\mcK (a,f)\|_{\C\times Y} &\lesssim \|(a,f)\|_{\C\times Y} \\
\|\mcK(a,f)\|_{\C\times X}&\lesssim \|(a,f)\|_{\C \times X} \\
 \|\hat{\mcK}(a,f)\|_{\C\times Y}&\lesssim \|(a,f)\|_{\C\times X}.
\end{align*}
\end{prop}

\begin{proof}
This follows from Corollaries 5.7 and 5.10 in \cite{DonKri12}.
\end{proof}

Now we can prove the existence of a solution to Eq.~\eqref{eq:sysFourier3}.

\begin{thm}
\label{thm:exsol}
The function $\underline{\Phi}$ as defined in Eq.~\eqref{eq:sysFourier3} maps the closed unit
ball of $\mcX^{\frac43-2\delta,\frac43-2\delta}$ to itself and is contractive. As a consequence,
there exists a unique solution $\mbx$ of Eq.~\eqref{eq:sysFourier3}
in the closed unit ball of $\mcX^{\frac43-2\delta,\frac43-2\delta}$.
\end{thm}

\begin{proof}
From Lemma \ref{lem:estH} and Proposition \ref{prop:N} we infer
\begin{align*}
\|\mcH\mcN(\mbx)-\mcH\mcN(\underline{y})\|_{\mcX^{\frac43-2\delta,\frac43-2\delta}}&\lesssim \tau_0^{-\epsilon}
\|\mbx-\underline{y}\|_{\mcX^{\frac43-2\delta,\frac43-2\delta}}
\end{align*}
for all $\mbx$, $\underline{y}$ in the unit ball in $\mcX^{\frac43-2\delta,\frac43-2\delta}$ and some $\epsilon>0$.
Thus, in view of Lemma \ref{lem:invmatrixHKD} and the contraction mapping principle it suffices
to prove smallness of the remaining terms in the appropriate spaces.
In the discussion following Lemma \ref{lem:inhomY} we have already noted that
\[ \|\mcH \mbe\|_{\mcX^{\frac43-2\delta,\frac43-2\delta}}\lesssim \nu^{2\delta}\tau_0^{-\epsilon} \]
for some $\epsilon>0$.
Furthermore, we have
\begin{align*}
\|\mcH \beta^2 \hat{\mcK}\mbx\|_{\mcX^{\frac43-2\delta,\frac43-2\delta}}&\lesssim \nu^{2\delta}\tau_0^{-1+2\delta}
\|\mbx\|_{\mcX^{\frac43-2\delta,\frac43-2\delta}}
\end{align*}
by Lemma \ref{lem:estH} and Proposition \ref{prop:K}.
\end{proof}

Finally, we consider the derivative $\hat{\mcD}\mbx$.

\begin{cor}
\label{cor:Dx}
Let $\mbx$ be the solution from Theorem \ref{thm:exsol}.
Then $\hat{\mcD}\mbx$ belongs to the closed unit ball of $\mcY^{\frac43,\frac43}$.
\end{cor}

\begin{proof}
The claim follows by noting that Eq.~\eqref{eq:sysFourier2}
implies
\[ \hat{\mcD}\mbx=(1+2\hat{\mcD}\mcH \beta \mcK)^{-1}\hat{\mcD}\mcH\left [
\mbe+\mcN(\mbx)-\beta^2 \hat{\mcK}\mbx \right ]\]
and by Propositions \ref{prop:N}, \ref{prop:K} as well as Lemmas \ref{lem:inhomY}, \ref{lem:estH},
\ref{lem:invmatrixDHK} we see that $\hat{\mcD}\mbx$ belongs to the unit ball of $\mcY^{\frac43,\frac43}$.
\end{proof}
In light of Lemma 4.3 in \cite{DonKri12}, we infer that
\[
\eps(\tau, R): = R^{-1}\big[x_d(\tau)\phi_d(R) + \int_0^\infty x(\tau, \xi)\phi(R, \xi)\rho(\xi)\,d\xi\big]
\]
satisfies $(\eps(\tau, \cdot), \eps_\tau(\tau,\cdot))\in H^{\frac{5}{4}}\times H^{\frac{1}{4}}$, with norm vanishing
as $\tau\rightarrow +\infty$.
We have thus proved

\begin{thm}Let $\lambda(t)$ be as in \eqref{Min0}. Then the equation \eqref{u5} admits a solution $u(t, r)$ of the form
\[
u(t, r) = \lambda^{\frac{1}{2}}(t)W(\lambda(t)r) + \eps(t, r),\,t\in (0, t_0),
\]
with $(\eps(t, \cdot), \eps_{t}(t, \cdot)) \in  H^{\frac{5}{4}}\times H^{\frac{1}{4}}$. Given $\delta>0$, one may arrange
\[
\|(\eps, \eps_t)\|_{(\dot{H}^1\times L^2)(r\geq t)}<\delta
\]
\end{thm}

\begin{proof}
This follows from Theorem~\ref{thm:exsol} as in \cite{KST} by exploiting energy conservation and smallness of energy outside the light cone.
\end{proof}

\bigskip

\centerline{\scshape Roland Donninger, Joachim Krieger}
\medskip
{\footnotesize
 \centerline{B\^{a}timent des Math\'ematiques, EPFL}
\centerline{Station 8,
CH-1015 Lausanne,
  Switzerland}
  \centerline{\email{joachim.krieger@epfl.ch}}
} 

\medskip

\centerline{\scshape Min Huang, Wilhelm Schlag}
\medskip
{\footnotesize
 \centerline{Department of Mathematics, The University of Chicago}
\centerline{5734 South University Avenue, Chicago, IL 60615, U.S.A.}
\centerline{\email{schlag@math.uchicago.edu}
}
} 


\begin{thebibliography}{10}

\bibitem{Biz}
\newblock P.\ Bizo\'n, T.\ Chmaj, Z.\ Tabor
\newblock  \emph{On blowup for semilinear wave equations with a focusing nonlinearity.}
\newblock Nonlinearity 17 (2004), no. 6, 2187--2201.

\bibitem{DonKri12}
\newblock R.\ Donninger, J.\ Krieger
\newblock \emph{Nonscattering solutions and blow up at infinity for the critical wave equation. }
\newblock preprint, arXiv: 1201.3258v1, to appear in Math.~Ann.

\bibitem{DKM1}
\newblock T.\ Duyckaerts, C.\  Kenig, F.\ Merle
\newblock \emph{ Universality of blow-up profile for small radial type II blow-up solutions of energy-critical wave equation.}
\newblock  J.\ Eur.\ Math.\ Soc., no. 3,   \textbf{13} (2011),  533--599.

\bibitem{DKM2}
\newblock T.\ Duyckaerts, C.\  Kenig, F.\ Merle
\newblock \emph{Universality of the blow-up profile for small type II blow-up solutions of energy-critical wave equation: the non-radial case.}
\newblock   J.\ Eur.\ Math.\ Soc.\ (JEMS) 14 (2012), no.~5, 1389--1454.

\bibitem{DKM3}
\newblock T.\ Duyckaerts, C.\  Kenig, F.\ Merle
\newblock \emph{ Profiles of bounded radial solutions of the focusing, energy-critical wave equation.}
\newblock preprint, arXiv:1201.4986, to appear in GAFA.

\bibitem{DKM4}
\newblock T.\ Duyckaerts, C.\  Kenig, F.\ Merle
\newblock \emph{Classification of radial solutions of the focusing, energy-critical wave equation.}
\newblock preprint, arXiv:1204.0031.

\bibitem{HR}
\newblock M.\ Hillairet, P.\ Rapha\"el
\newblock {\em Smooth type II blow up solutions to the four dimensional energy critical wave equation.}
\newblock preprint, arxiv:1010.1768

\bibitem{KS}
\newblock  J.\ Krieger, W.\ Schlag
\newblock {\em On the focusing critical semi-linear wave equation.}
\newblock Amer.\ J.\ Math.\ 129 (2007), no.\ 3, 843--913.


\bibitem{KrSch}
\newblock J.\ Krieger,   W.\  Schlag,
\newblock \emph{ Full range of blow up exponents for the quintic wave equation in three dimensions.}
\newblock Preprint 2012.

\bibitem{KST0}
\newblock J.\ Krieger,   W.\  Schlag, D.\ Tataru
\newblock \emph{  Renormalization and blow up for charge one equivariant critical wave maps.}
\newblock Invent.\ Math.\ 171 (2008), no.\ 3, 543--615.

\bibitem{KST}
\newblock J.\ Krieger,   W.\  Schlag, D.\ Tataru
\newblock \emph{   Slow blow-up solutions for the $H^1(\R^3)$ critical focusing semilinear wave equation.}
\newblock Duke Math.\ J., no.~1, \textbf{ 147}  (2009),   1--53.

\bibitem{KST2}
\newblock J.\ Krieger,   W.\  Schlag, D.\ Tataru
\newblock \emph{   Renormalization and blow up for the critical Yang-Mills problem.}
\newblock Adv.\ Math.\ 221 (2009), no.~5, 1445--1521. 

\bibitem{RR}
\newblock P.\ Rapha\"el, I.,  Rodnianski
\newblock {\em Stable blow up dynamics for the critical co-rotational wave maps and equivariant Yang-Mills problems.}
\newblock Publ.\ Math.\ Inst.\ Hautes \'Etudes Sci.\ 115 (2012), 1--122.



\end{thebibliography}
\end{document}